\documentclass[11pt]{amsart}

\usepackage[a4paper,inner=3.6cm,outer=3.6cm,top=3.7cm,bottom=3.7cm]{geometry}
\usepackage{amsfonts}
\usepackage{amssymb}
\usepackage{amsmath}
\usepackage[dvipdfm]{graphics}
\usepackage[all,cmtip,2cell]{xy}
\usepackage{comment}
\usepackage{array}
\usepackage{enumitem}
\usepackage{mathtools}
\setlist[itemize]{leftmargin=*}
\setlist[enumerate]{leftmargin=*}

\newcommand{\C}{\mathbb{C}}
\newcommand{\Z}{\mathbb{Z}}
\newcommand{\N}{\mathbb{N}}
\newcommand{\rk}{\mathrm{rk}\,}
\newcommand{\id}{\mathrm{Id}}
\newcommand{\R}{\mathbb{R}}

\newcommand{\cO}{\mathcal{O}}
\renewcommand{\P}{\mathbb{P}}
\newcommand{\bd}{\partial}
\newcommand{\ai}{$A_\infty$\ }
\newcommand{\F}{\mathcal{F}}
\newcommand{\ii}{\iota}
\newcommand{\M}{\mathcal{M}}
\newcommand{\cL}{\mathcal{L}}
\renewcommand{\a}{\"a}
\newcommand{\p}{\phi}

\newcommand{\fix}{\mathrm{Fix\,}}
\newcommand{\ps}{_{\mathrm{push}}}
\newcommand{\gl}{\mathrm{gl}}
\newcommand{\hf}{_{\mathrm{floer}}}
\newcommand{\Symp}{\text{\it Symp}}
\newcommand{\Ham}{\text{\it Ham}}
\newcommand{\Diff}{\text{\it Diff\,}}
\newcommand{\Crit}{\text{\it Crit\,}}
\newcolumntype{C}[1]{>{\centering\arraybackslash$}p{#1}<{$}}

\newcommand{\dist}{\text{\it dist}}
\newcommand{\sgn}{\mathrm{sign}\,}

\newcommand{\half}{\frac 1 2}
\DeclareMathOperator{\spn}{span}
\def\co{\colon\thinspace}

\newtheorem{theorem}{Theorem}[section]
\newtheorem{proposition}[theorem]{Proposition}
\newtheorem{lemma}[theorem]{Lemma}
\newtheorem{corollary}[theorem]{Corollary}
\theoremstyle{definition}
\newtheorem{definition}[theorem]{Definition}
\newtheorem{remark}[theorem]{Remark}
\newtheorem{hypothesis}[theorem]{Hypothesis}

\numberwithin{equation}{section}

\begin{document}

\title[Commuting symplectomorphisms]{Commuting symplectomorphisms and \\ Dehn twists in divisors}
\author{Dmitry Tonkonog}
\email{dt385@cam.ac.uk, dtonkonog@gmail.com}
\address{Department of Pure Mathematics and Mathematical Statistics,
University of Cambridge,
Wilberforce Road,
Cambridge,
CB3 0WB, UK}

\begin{abstract}
Two commuting symplectomorphisms of a symplectic manifold give rise to actions on Floer cohomologies of each other.
We prove the elliptic relation saying that the supertraces of these two actions are equal.

In the case when a symplectomorphism $f$ commutes with a symplectic involution, 
the elliptic relation provides a lower bound on the dimension of $HF^*(f)$ in terms of the Lefschetz number of $f$ restricted to the fixed locus of the involution.
We apply this bound to prove that Dehn twists around vanishing Lagrangian spheres inside most hypersurfaces in Grassmannians have infinite order in the symplectic mapping class group.
\end{abstract}

\maketitle
\setcounter{tocdepth}{1}

{\small
\tableofcontents
}

\newpage

\begin{section}{Introduction and main results}

\subsection{Overview}

Let $X$ be a symplectic manifold 
and $\Symp(X)/\Ham(X)$ be the group of all symplectomorphisms of $X$ modulo Hamiltonian isotopy. When $X$ is simply-connected, this group is the same as $\pi_0 Symp(X)$.
If one denotes by $\pi_0 \Diff(X)$ the smooth mapping class group,
there is an obvious forgetful map
$$
\Symp(X)/\Ham(X)\quad
\xrightarrow{\makebox[1.5cm]{{\footnotesize \it forgetful}}}
\quad
 \pi_0 \Diff(X).
$$
Paul Seidel in his thesis \cite{SeiThesis} found examples when this map is not injective:
if $X$ is any complete intersection of complex dimension 2 other than $\P^2$ or $\P^1 \times \P^1$, and $\tau\co X\to X$ is a certain symplectomorphism called the Dehn twist, then $\tau^2$ is smoothly isotopic to the identity, but not Hamiltonian isotopic to the identity.
Later Seidel proved  \cite{Sei00} that the kernel of the above map is infinite for some K3 surfaces, again by considering the group generated by a Dehn twist. Using a new technique, we study Dehn twists 
in certain divisors (the main examples are divisors in Grassmannians) and
extend the range of examples when the above forgetful map has infinite kernel.

Suppose $X$ satisfies the so-called $W^+$ condition, which is slighly stronger than weak monotonicity.
We define, for two commuting symplectomorphisms $f,g\co  X\to X$,
their actions on Floer cohomology $f\hf\co HF^*(g)\to HF^*(g)$, $g\hf\co HF^*(f)\to HF^*(f)$.
We then prove a theorem which was proposed by Paul Seidel, cf.~\cite[Remark~4.1]{SeiFlux}, who suggested it be called the 
elliptic relation.

\begin{theorem}[Elliptic relation]
\label{th:ell_rel}
If $X$ is a symplectic manifold satisfying the $W^+$ condition and $f,g\co X\to X$
are two commuting symplectomorphisms, then
$$
STr(f\hf)\ =\ STr(g\hf) \in \Lambda.
$$
\end{theorem}

Here $\Lambda$ is the Novikov field. In the rest of the introduction, we explain the elliptic relation, state its Lagrangian version, and
consider applications to Dehn twists in divisors. We begin by discussing our results regarding Dehn twists.

\subsection{Order of Dehn twists in divisors}
Let $Gr(k,n)$ be the Grassmannian of $k$-planes in $\C^n$. Let $\cO(d)$ be the line bundle on
$Gr(k,n)$ which is the pullback of $\cO_{\P^N}(d)$ under the Pl\" ucker embedding $Gr(k,n)\subset \P^N$.
Consider  a smooth divisor $X\subset Gr(k,n)$ in the linear system $|\cO(d)|=\P H^0(Gr(k,n),\cO(d))$.
The results below are interesting even for $Gr(1,n)=\P^{n-1}$, so for simplicity one can take  $X\subset \P^{n-1}$ to be a smooth projective hypersurface of degree $d$ throughout this subsection.

For $d\ge 2$, $X$ contains a class of Lagrangian spheres  which we call $|\cO(d)|$-vanishing Lagrangian spheres, which, briefly,
 are vanishing cycles for algebraic degenerations of $X$ inside the linear system $|\cO(d)|$.
To every parametrised  Lagrangian sphere $L\subset X$ one associates a symplectomorphism $\tau_L\co X\to X$ called the Dehn twist
around $L$. (The definitions are given in Section~\ref{sec:van_spheres_intro}.)
We prove the following.

\begin{theorem}
\label{th:twist_inf_grass}
Let $X\subset Gr(k,n)$ be a smooth divisor in the linear system $|\cO(d)|$, and $L\subset X$ be an $|\cO(d)|$-vanishing Lagrangian sphere. Suppose
$$
3\le d\le n\qquad\mathrm{or}\qquad d\ge k(n-k)+n-2. 
$$
Then the Hamiltonian isotopy class of $\tau_L$ is an element of infinite order in the group $Symp(X)/Ham(X)$.
\end{theorem}

When $d=2$ and $k=1$ ($X$ is a projective quadric), $\tau_L$ has order 1 or 2 depending on the parity of $n$ \cite[Lemma 4.2]{Smith12}. While our proof crucially uses $d\ge 3$, further restrictions on $d$ 
are only needed to make $X$ satisfy the $W^+$ condition, so that the ``classical'' definition of Floer
cohomology of symplectomorphisms $X\to X$ applies.
There are techniques \cite{FO99} defining Floer cohomology of symplectomorphisms on arbitrary symplectic manifolds. 
With their help the proof of Theorem~\ref{th:twist_inf_grass} (and of Theorem~\ref{th:ell_rel}) should work for all $d\ge 3$.

Recall the forgetful map $\Symp(X)/\Ham(X)\to \pi_0 \Diff(X)$.
If $\dim_\C X$ is odd and $d\ge 3$, the image of $\tau_L$ has infinite order in $\pi_0 \Diff(X)$
by the Picard-Lefschetz formula, so Theorem~\ref{th:twist_inf_grass} becomes trivial. However, when $\dim_\C X$ is even, the image of $\tau_L$ has finite order in
$\pi_0 \Diff(X)$ (see Subsection~\ref{subsec:dehn_twists} for details), so Theorem~\ref{th:twist_inf_grass} is really of symplectic nature in this case.
When $X$ is Calabi-Yau ($d=n$), Theorem~\ref{th:twist_inf_grass} follows from a grading argument of Paul Seidel
\cite{Sei00}.
Theorem~\ref{th:twist_inf_grass} is new in all cases when $\dim_\C X$ is even and $d\neq n$.
For instance, it appears to be new even for the cubic surface $X\subset \P^3$.

Let 
$$\Delta\subset \P H^0(Gr(k,n),\cO(d))$$ 
be the discriminant variety parameterising all singular divisors
in $|\cO(d)|$.
Theorem~\ref{th:twist_inf_grass} implies a corollary about the fundamental group of the complement
to the discriminant.
Fix a divisor $X\in|\cO(d)|$. For any family $X_t\subset Gr(k,n)$ of smooth divisors in $|\cO(d)|$, $t\in[0,1]$,  there is a symplectic parallel transport map, a symplectomorphism $X_0\to X_1$ which depends 
up to Hamiltonian isotopy only on the homotopy class of the path $X_t$ relative to its endpoints.
Applied to loops, parallel transport gives the symplectic monodromy map
$$
 \pi_1\left(\P H^0(Gr(k,n),\cO(d))\setminus \Delta\right)\quad 
\xrightarrow{\makebox[2cm]{{\footnotesize \it monodromy}}}
\quad Symp(X)/Ham(X). 
$$
The discriminant complement contains a distinguished conjugacy class of loops $\gamma$ called meridian loops.
A meridian loop 
$$\gamma\subset \P H^0(Gr(k,n),\cO(d))\setminus \Delta$$ 
is the boundary of a 2-disk in $\P H^0(Gr(n,k),\cO(d))$ that intersects $\Delta$ transversely once.
The image of such a loop under the monodromy map is the Dehn twist $\tau_L$ where $L\subset X$
is an $|\cO(d)|$-vanishing Lagrangian sphere. Theorem~\ref{th:twist_inf_grass} implies the following.

\begin{corollary}
\label{cor:fund_gp_grass}
If $3\le d\le n$ or $d\ge k(n-k)+n-2$,  and
$\gamma\subset \P H^0(Gr(k,n),\cO(d))\setminus \Delta$
is a meridian loop,
then
 $$
[\gamma]\in \pi_1\left(\P H^0(Gr(k,n),\cO(d))\setminus \Delta\right) \quad \text{is\ an\ element\ of\ infinite\ order.}
$$
\end{corollary}

Note that $[\gamma]\in H_1\left(\P H^0(Gr(k,n),\cO(d))\setminus \Delta;\Z\right)$ has finite order.
For the projective space $Gr(1,n)=\P^{n-1}$, the fundamental group
$\pi_1(\P H^0(\P^{n-1},\cO(d))\setminus \Delta)$ is computed by L\"onne in \cite{Loe09}
and implies Corollary~\ref{cor:fund_gp_grass} for $k=1$. For $k\neq 1$,
 the corresponding fundamental group seems not to be studied, but
 Corollary~\ref{cor:fund_gp_grass} should allow a more straightforward proof, suggested to us by Dmitri Panov. Namely,
assume $\dim_\C X$ is even (otherwise the corollary follows from the fact the Dehn twist has infinite order topologically) and
consider the $d:1$ cover of $Gr(k,n)$ branched along $X$, which now has odd complex dimension. A nodal degeneration of $X$ provides an $A_d$-degeneration of the cover, and the monodromy around such a degeneration, which is a composition of Dehn twists around a chain of Lagrangian spheres, has infinite order in the smooth mapping class group (which uses the Picard-Lefschetz formula and the fact the spheres are now odd-dimensional). This observation is enough to imply Corollary~\ref{cor:fund_gp_grass}, bypassing the need to consider the Dehn twist in $X$ itself. However, we decided to keep Corollary~\ref{cor:fund_gp_grass} to add an additional context to the main theorems.

We prove analogues of 
Theorem~\ref{th:twist_inf_grass} 
and Corollary~\ref{cor:fund_gp_grass} for divisors in some very ample line bundles $\cL\to Y$, where
 $Y$ is a K\" ahler manifold which carries a holomorphic involution
with certain properties. The precise statement is postponed to Subsection~\ref{subsec:gen_theorems}.

\subsection{Elliptic relation for commuting symplectomorphisms}
To prove Theorem~\ref{th:twist_inf_grass}, we use the elliptic relation 
(Theorem~\ref{th:ell_rel})
which we now discuss.

Let $X$ be a symplectic manifold satisfying the $W^+$ condition explained in Section~\ref{sec:ell_rel}; for example, $X$ can be a K\" ahler manifold which is either Fano, or whose canonical class $K_X$ is sufficiently positive. Given a symplectomorphism $f\co X\to X$, one defines its Floer cohomology $HF^*(f)$. It is a $\Z_2$-graded vector space, $HF^*(f)=HF^0(f)\oplus HF^1(f)$, over the Novikov field 
$$\Lambda=\left\{\sum_{i=0}^{\infty}a_iq^{\omega_i}:\ 
a_i\in \C,\ \omega_i\in\R,\ \lim_{i\to\infty}\omega_i=+\infty\right\}.$$

For any two commuting symplectomorphisms $f,g\co X\to X$ we define invertible automorphisms
$$g\hf\co HF^*(f)\to HF^*(f) \quad\text{and}\quad f\hf \co HF^*(g)\to HF^*(g).$$
The construction of $HF^*(f)$ uses a time-dependent almost complex structure $J$ and a Hamiltonian $H$ to define a vector space $HF^*(f;J,H)$. This vector space is canonically isomorphic (on the chain level) to $HF^*(gfg^{-1}; g^*J, H\circ g)$ by composing all pseudo-holomorphic curves with $g$. If $f,g$ commute, $g\hf$ is the composition of isomorphisms
$$
HF^*(f;J,H)\longrightarrow HF^*(gfg^{-1}; g^*J, H\circ g) = HF^*(f; g^*J, H\circ g)  \longrightarrow HF^*(f;J,H)
$$
where the last arrow is the continuation map associated to a homotopy of data from $(g^*J,H\circ g)$ to $(J,H)$.

The automorphisms $f\hf,g\hf$ have zero degree, and one can define their supertrace:
$$
STr(g\hf)\coloneqq Tr(g\hf|_{HF^{0}(f)})-Tr(g\hf|_{HF^{1}(f)}) \ \in \ \Lambda.
$$
Recall that Theorem~\ref{th:ell_rel} asserts that $STr(f\hf)=STr(g\hf)$.

Now suppose a symplectomorphism $f$ commutes with a finite-order symplectomorphism $\phi$,
$\phi^k=\id$, with fixed locus $X^\phi$. Then $X^\phi$ is a disjoint union of symplectic submanifolds.
Using an argument reminiscent of the PSS isomorphism, we show that
$$STr(f\hf\co  HF^*(\phi)\to HF^*(\phi))=L(f|_{X^\phi})\cdot q^0.$$ 
The right hand side is  the topological Lefschetz number 
$$L(f|_{X^\phi})=Tr(f^*|_{H^{even}(X^\phi)}-Tr(f^*|_{H^{odd}(X^\phi)})$$
where $f^*\co  H^*(X^\phi)\to H^*(X^\phi)$ is the classical action on the cohomology of $X^\phi$.
On the other hand, using that $\phi$ has finite order, we show that
$STr(\phi\hf\co  HF^*(f)\to HF^*(f))$ equals $a\cdot q^0$ where
$|a|\le \dim_\Lambda HF^*(f)$.
Combining this with the elliptic relation, we obtain the following corollary.

\begin{proposition}
\label{prop:hf_bound_fixp_notrans}
Let $X$ be a symplectic manifold satisfying the $W^+$ condition, $f,\p\co X\to X$ two commuting symplectomorphisms
and $\p^k=\id$.
Then $$\dim_\Lambda HF^*(f)\ge \left|L(f|_{X^\phi})\right|.$$
\end{proposition}

\begin{remark}
The fixed locus $X^\phi$ is allowed to be disconnected, with components of different dimensions.
\end{remark}

\begin{remark}
If $f\co X\to X$ is a diffeomorphism with smooth fixed locus $X^f$, such that $\id-df(x)|_{N_x\Sigma}$ is non-degenerate on the normal space $N_x\Sigma\subset T_xX$ to any connected component $\Sigma\subset X^f$ for every $x\in \Sigma$,
then 
$$L(f)=\sum_{\Sigma\subset X^f}\sgn(\det (\id-df|_{N_x\Sigma}))\cdot \chi(\Sigma).$$ 
Consequently, if $\phi,\psi\co X\to X$ are finite order symplectomorphisms, we get
$L(\phi|_{X^\psi})=L(\psi|_{X^\phi})=\chi(X^\phi\cap X^\psi)$, provided the latter intersection is clean. This agrees with the elliptic relation and the topological interpretation of the Floer-homological actions for finite order maps.
\end{remark}

\begin{remark}
\label{rem:alt_proof_bound_intro}
It is possible to give a more straightforward proof of Proposition~\ref{prop:hf_bound_fixp_notrans} 
which does not appeal to Theorem~\ref{th:ell_rel}, but still requires some analysis
in the spirit of \cite[Lemma~14.11]{SeiBook08}. See Remark~\ref{rem:alt_proof_bound_main} for more details.
\end{remark}

\begin{remark}
Theorem~\ref{th:ell_rel} holds  when $f,g$ commute only up to Hamiltonian isotopy, and more generally when $fg^{-1}$ is isomorphic to $\id$ in the Donaldson category, whose objects are symplectomorphisms of $X$ and $Hom(f,g)=HF^*(fg^{-1})$; the proofs require only minor modifications.
In Proposition~\ref{prop:hf_bound_fixp_notrans}, $f,g$ can also be allowed to commute up to Hamiltonian isotopy. 
\end{remark}

\subsection{Outline of proof of Theorem~\ref{th:ell_rel}}
The complete proof of Theorem~\ref{th:ell_rel} with all necessary definitions is found in Section~\ref{sec:ell_rel};
here we provide a sketch, illustrated by Figure~\ref{fig:el_rel}, and indicate the main technical issue we have to solve.

\begin{figure}[h]
\includegraphics{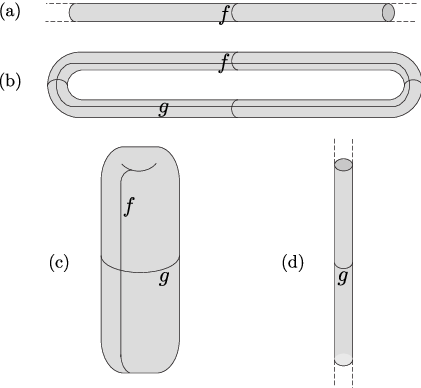}
\caption{Changing the base of a symplectic fibration in the proof of Theorem~\ref{th:ell_rel}.}
\label{fig:el_rel}
\end{figure}

Let  $f,g$ be two commuting symplectomorphisms. By our definition, the supertrace $STr(g\hf)$
is computed by counting certain solutions to  Floer's continuation equation, or equivalently by counting holomorphic sections of a certain symplectic fibration $E_f\to S^1\times \R$, see Figure~\ref{fig:el_rel}(a). This  fibration has monodromy $f$ along $S^1$, and the almost complex structure on $E_f$  differs by the action of $g$ over the two ends of the cylinder. We count only those sections whose asymptotics differ by the action of $g$ over the ends of the cylinder. One can therefore glue the fibration, together with the almost complex structure, into a fibration $E_{f,g}\to S^1\times S^1$. A gluing theorem in Symplectic Field Theory gives a bijection between holomorphic sections $S^1\times \R\to E_f$ (with asymptotics as above) and all holomorphic sections $S^1\times S^1\to E_{f,g}$ where $S^1\times S^1$ is endowed with the complex structure which is very ``long'' in the direction of the second $S^1$-factor: see Figure~\ref{fig:el_rel}(b). We will refer to this bijection by $(*)$ in the next few paragraphs.

On the other hand, the count of holomorphic sections  $S^1\times S^1\to E_{f,g}$ does not depend on the chosen complex structure on $S^1\times S^1$. 
Take another complex structure on $S^1\times S^1$ which is ``long'' in the first $S^1$-factor instead of the second one, see Figure~\ref{fig:el_rel}(c). The same gluing argument as above $(*)$ implies that the count of holomorphic sections $S^1\times S^1\to E_{f,g}$ is equal
to the count of holomorphic sections $\R\times S^1\to E_g$ (with asymptotics different by the action of $f$ over the ends of the cylinder), where $E_g\to \R\times S^1$ is the fibration obtained by cutting $E_{f,g}$ along the first $S^1$-factor,
see Figure~\ref{fig:el_rel}(d). Similarly to what we began with, the latter count of holomorphic sections over $\R\times S^1$ gives $STr(f\hf)$.


The key difficulty in upgrading this sketch to a proof is to determine how the bijection $(*)$ behaves with respect to the signs attached
to  sections over the cylinder (which in general depend on the choice of a ``coherent orientation'', but are canonical for sections contributing to the supertrace),
and signs canonically attached to sections over the torus. The outcome is that $(*)$
multiplies  signs by $(-1)^{\deg x}$ where $x$ is a $\pm\infty$ asymptotic periodic orbit of the section over the cylinder. (The $\pm \infty$ asymptotics differ by $g$ and thus have the same degree.)
This is Formula (\ref{eq:sign_compare}) in Section~\ref{sec:ell_rel}.
It explains why Theorem~\ref{th:ell_rel} is an equality between  supertraces and
not usual traces.
(We have not found Formula (\ref{eq:sign_compare}) elsewhere in the literature. Coherent orientations in SFT are discussed in \cite{EliGiHo10, BouMo04},
see especially \cite[Corollary 7]{BouMo04}, but don't seem to give the result we need).

\begin{remark}
 \label{rem:elliptic_categorical}
As the proof uses the torus with different complex structures (i.e.~elliptic curves), this justifies the name ``elliptic relation''. There is some categorical perspective to the elliptic relation, as well: Ben-Zvi and~Nadler \cite[Theorem~1.2]{BZNa13} obtained an equality between the so-called ``secondary traces'' in a 2-category, which also comes from cutting the torus into pieces in two different ways (however, not into two different cylinders as we do).
\end{remark}

\subsection{Elliptic relation for invariant Lagrangians}
Before explaining how the elliptic relation helps to prove  Theorem~\ref{th:twist_inf_grass}, let us 
discuss its Lagrangian version. The coefficient field is still $\Lambda$.
Definitions and sketch proofs are briefly presented in Subsection~\ref{subsec:ell_rel_lag}.

Let $X$ be a connected monotone symplectic manifold (e.g.~complex Fano variety), and $L_1,L_2\subset X$ monotone Lagrangians (e.g.~simply connected). Suppose there is a symplectomorphism $\phi\co X\to X$ such that $\phi(L_1)=L_1$, $\phi(L_2)=L_2$. Under a condition involving spin structures, formulated later as Hypothesis~\ref{hyp:equiv_spin},
a version of the open-closed string map provides twisted cohomology classes $[L_1]^\phi\in HF^*(\phi)$, $[L_2]^{\phi^{-1}}\in HF^*(\phi^{-1})$.
Consider the quantum product $[L_1]^\phi *[L_2]^{\phi^{-1}}\in QH^*(X)$ and 
the  map  $\chi\co QH^*(X)\to \Lambda$ which is the integration over $[X]$ (sending the volume form to $1$ and all elements of $H^{<2n}(X)$, seen as elements of $QH^*(X)$, to 0).
Under the assumptions of the next theorem, there is again an action
$\phi\hf\co  HF^*(L_1,L_2)\to HF^*(L_1,L_2)$, with Floer cohomology taken over $\Lambda$.

\begin{theorem}[Elliptic relation]
\label{th:ell_rel_Lag}
Suppose $(X,L_1,L_2)$ are monotone, $\phi\co X\to X$ is a symplectomorphism, $\phi(L_i)=L_i$. If the base field has $\mathrm{char}\neq 2$, suppose the $L_i$ are orientable and Hypothesis~\ref{hyp:equiv_spin} is satisfied (e.g.~the $L_i$ are simply-connected). Then
$$
STr(\phi\hf)=\chi\left([L_1]^\phi *[L_2]^{\phi^{-1}}\right).
$$ 
\end{theorem}

If $\phi^k=\id$ and the fixed loci $L_i^\phi\subset X^\phi$ are smooth and orientable, the $q^0$-term of the right hand side equals
the classical homological intersection $[L_1^\phi]\cdot[L_2^\phi]\in \Z$ inside $X^\phi$, where $[L_i^\phi ]\in H_{\dim_\R X/2}(X;\Z)$.
On the other hand,  eigenvalue decomposition of $\phi\hf$ implies that the left hand side equals $a\cdot q^0$ with $a\in \C$,  $|a|\le\dim_\Lambda HF^*(L_1,L_2)$. The elliptic relation yields the following analogue of Proposition~\ref{prop:hf_bound_fixp_notrans}.

\begin{proposition}
 \label{prop:hf_bound_lag}
Under the assumptions of Theorem~\ref{th:ell_rel_Lag}, if $\phi^k=\id$ and the fixed loci $L_i^\phi,X^\phi$ are smooth and orientable then
$$
\dim_\Lambda HF^*(L_1,L_2)\ge \left| [L_1^\phi]\cdot[L_2^\phi]\right|.
$$
\end{proposition}

As our Lagrangians are monotone, we can pass from $\Lambda$-coefficients to the base field (e.g.~$\C$ or $\Z/2\Z$) without changing the dimensions of Floer cohomology \cite[Remark~4.4]{WeWo09}. So Proposition~\ref{prop:hf_bound_lag} gives the same bound on $\dim  HF^*(L_1,L_2;\C)$ or $\dim HF^*(L_1,L_2;\Z/2\Z)$. However,
the proof of  Proposition~\ref{prop:hf_bound_lag} crucially uses Theorem~\ref{th:ell_rel_Lag} over $\Lambda$, as can  be seen from the sketch we presented.

As a simple application of Proposition~\ref{prop:hf_bound_lag} we can recover the following known fact: $\R\P^n\subset \C\P^n$ is not self-displaceable by a Hamiltonian isotopy, as $\dim HF^*(\R\P^n,\R\P^n;\allowbreak \Z/2\Z)\ge 1$. When $n$ is even, this is true because the Euler characteristic of $\R\P^n$ equals 1 over a characteristic 2 field. When $n$ is odd, consider the hyperplane reflection $\ii$ on $\C\P^n$ so that $(\R\P^n)^\ii=\R\P^{n-1}$ and apply Proposition~\ref{prop:hf_bound_lag}. 

In Appendix~\ref{app:growth_ring} we provide a more interesting application of  Proposition~\ref{prop:hf_bound_lag}.
Namely, we prove that for $L\subset X$ as in Theorem~\ref{th:twist_inf_grass}, and if $X$ is in addition Fano and even-dimensional,
there is an isomorphism of rings $HF^*(L,L;\C)\cong \C[x]/x^2$.
For Lagrangian spheres in the cubic surface, this was proved by Sheridan~\cite{She13}, and after the present paper had appeared, it was observed by Biran and Membrez~\cite[Subsection~1.3.2]{BM14} that for a Lagrangian sphere in a projective hypersurface, which is Fano and of degree at least 3, the isomorphism $HF^*(L,L;\C)\cong \C[x]/x^2$ follows from the known structure of $QH^*(X)$, regardless of the complex dimension of $X$. Our method is completely different: it does not use any knowledge of $QH^*(X)$, and works for hypersurfaces in Grassmannians as well as in some more abstract cases discussed in Appendix~\ref{app:growth_ring}.

\begin{remark}
The action $\phi\hf$ on $HF^*(L_1,L_2)$ (as well the actions in the case of two commuting symplectomorphisms) can be  defined using functors coming from Lagrangian correspondences \cite{WeWo10a,WeWo10b}. It is possible that the two versions of the elliptic relation admit a generalisation for  Lagrangian correspondences.
\end{remark}

\subsection{Outline of proof of Theorem~\ref{th:twist_inf_grass}}
We have already mentioned that Theorem~\ref{th:twist_inf_grass} holds for topological reasons when $\dim X$ is odd. Suppose therefore that $\dim_\C Gr(k,n)$ is odd, so that $\dim_\C X$ is even. The Grassmannian has an involution $\ii$ whose fixed locus contains an even-dimensional connected component $\tilde \Sigma\subset Gr(k,n)$. For example, when $k=1$ we can take the involution $(x_1:x_2:x_3:x_4 : \ldots : x_n)\mapsto ({-x_1}:{-x_2}:
{-x_3} :x_4:\ldots :x_n)$ and $\tilde \Sigma=\P^2(x_1:x_2:x_3)$.

The key idea of reducing Theorem~\ref{th:twist_inf_grass} to Proposition~\ref{prop:hf_bound_fixp_notrans} is the following construction performed in Section~\ref{sec:van_sph_construct}. We construct a smooth divisor $X\subset Gr(k,n)$ invariant under $\ii$ such that the fixed locus $X^\ii$ of the involution $\ii|_X$ contains an odd-dimensional connected component $\Sigma=\tilde \Sigma\cap X$.
Next, we construct two $\ii$-invariant $|\cO(d)|$-vanishing Lagrangian spheres $L_1,L_2\subset X$ which intersect each other transversely once. Moreover, the fixed loci $L_i^\ii\coloneqq L_i\cap \Sigma$, $i=1,2$, are Lagrangian spheres in $\Sigma$ which intersect each other transversely once, see Figure~2.
This is where we need $d\ge 3$.

\begin{figure}[h]
\includegraphics{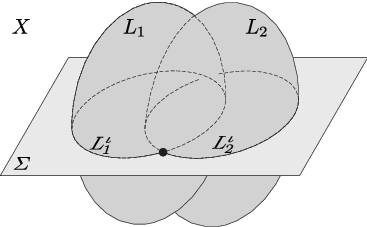}
\label{fig:InvSpheresFig}
\caption{Invariant Lagrangian spheres $L_1$ and $L_2$ used in the proof of Theorem~\ref{th:twist_inf_grass}.}
\end{figure}

Theorem~\ref{th:twist_inf_grass} is proved in Section~\ref{sec:proofs_spheres_divs}.
Consider the product of iterated Dehn twists $\tau_{L_1}^{2k}\tau_{L_2}^{2k}$. Because $L_1,L_2$ are $\ii$-invariant, $\tau_{L_1}^{2k}\tau_{L_2}^{2k}$ can be made $\ii$-equivariant. The Lefschetz number of $(\tau_{L_1}^{2k}\tau_{L_2}^{2k})|_{\Sigma}=\tau_{L_1^\ii}^{2k}\tau_{L_2^\ii}^{2k}$ on $\Sigma$ is equal to $c-4k^2$, where $c$ is a constant. This follows from the Picard-Lefschetz formula and crucially uses  the fact $\dim \Sigma$ is odd. If $\dim \Sigma$ were even, the trace would be independent of $k$. 
Consequently by Proposition~\ref{prop:hf_bound_fixp_notrans}, $\dim HF^*(\tau_{L_1}^{2k}\tau_{L_2}^{2k})$ grows  with $k$. 

Finally we note that $L_1,L_2$ from our construction can be taken one to another by a symplectomorphism of $X$. This means $\tau_{L_1}$ and $\tau_{L_2}$ are conjugate. If $\tau_{L_1}^{2k}$ was Hamiltonian isotopic to $\id$, then so would be $\tau_{L_2}^{2k}$ and
the product $\tau_{L_1}^{2k}\tau_{L_2}^{2k}$. This contradicts  the growth of Floer cohomology from above,
and proves Theorem~\ref{th:twist_inf_grass} for the 
specially constructed
$|\cO(d)|$-vanishing Lagrangian sphere $L_1\subset X$. 
If $X'$ is another smooth divisor linearly equivalent to $X$ and $L'\subset X'$ is
 another $|\cO(d)|$-vanishing Lagrangian sphere, Lemma~\ref{lem:van_unique}
says there is a symplectomorphism $X\to X'$ taking $L$ to $L'$.
This implies  Theorem~\ref{th:twist_inf_grass} in general.


\subsection{An extension of Theorem~\ref{th:twist_inf_grass}}
\label{subsec:gen_theorems}
Theorem~\ref{th:twist_inf_grass} is a particular case of the more general, but also more technical 
theorem which we now state.
Let $\cL$ be a very ample line bundle over a K\" ahler manifold $Y$.
It gives an embedding $Y\subset \P^N\coloneqq \P H^0(Y,\cL)^*$.

Suppose $\ii\co Y\to Y$ is a holomorphic involution which lifts to an automorphism of $\cL$.
The map $\ii$ induces a linear involution on $H^0(Y,\cL)^*$, splitting it into the direct sum of the $\pm 1$
 eigenspaces   $H^0(Y,\cL)_\pm^*$. Let $\Pi_\pm\subset \P^N$ be the projectivisations of these eigenspaces.
The fixed locus $Y^\ii\subset Y$ of the involution $\ii$ is:
$$Y^\ii=(\Pi_+\sqcup \Pi_-)\cap Y,$$ 
where the intersection is taken inside $\P^N$.
It is automatically smooth, but can have many connected components because the intersections $\Pi_+\cap Y$, $\Pi_-\cap Y$
may be disconnected.

\begin{theorem}
\label{th:twist_inf_general}
Under the above notation and assumptions,
fix $d\ge 3$ and
let $H^0(Y,\cL^{\otimes d})_\pm$ denote the $\pm 1$-eigenspace of the involution on $H^0(Y,\cL^{\otimes d})$ induced by $\ii$. Further, suppose one of the following:
\begin{enumerate}
 \item[(a)] $d$ is even, and

\noindent
$Y^\ii$ contains a connected component $\tilde \Sigma$ such that $\dim_\C\tilde \Sigma$ is even;
\item[(b)] $d$ is odd, 

\noindent
there is a smooth divisor in the linear system $\P H^0(Y,\cL^{\otimes d})_+$, and

\noindent
$\Pi_+\cap Y$ contains a connected component $\tilde \Sigma$ such that $\dim_\C\tilde \Sigma$ is even.
\end{enumerate}
Let $X\subset Y$ be a smooth divisor in the linear system $|\cL^{\otimes d}|$ and $L\subset X$ an $|\cL^{\otimes d}|$-vanishing Lagrangian sphere.
Denote by $\tau_L$ the Dehn twist around $L$, and
assume $X$ satisfies the $W^+$ condition.
Then the Hamiltonian isotopy class of $\tau_L$ is an element of infinite order in the group $Symp(X)/Ham(X)$.
The same is true if we replace  symbols $+$ with symbols $-$ in Case (b).
\end{theorem}

Like Theorem~\ref{th:twist_inf_grass}, Theorem~\ref{th:twist_inf_general} is new when $\dim_\C X$ is even and $X$ is not Calabi-Yau.

In Case (a), the existence of a smooth $\ii$-invariant divisor $X$ follows from Bertini's theorem, so it is not included as a condition of the theorem.
In Case~(b), an invariant divisor can sometimes be found using a strong Bertini theorem  \cite[Corollary~2.4]{DiHa91}, which gives the following.

\begin{lemma}
\label{lem:exist_smooth_inv_div}
Under the conditions of Theorem~\ref{th:twist_inf_general}, let $d$ be odd.
There is a smooth divisor in the linear system $\P H^0(Y,\cL^{\otimes d})_\pm$
if
every connected component of $\Pi_{\mp}\cap Y$ has dimension less than $\half \dim Y$.
\end{lemma}

As in the beginning of the introduction, we have the following corollary.
\begin{corollary}
\label{cor:fund_gp_general}
Under conditions of Theorem~\ref{th:twist_inf_general}, let
$\gamma\subset \P H^0(Y,\cL^{\otimes d})\setminus \Delta$
be a meridian loop, defined analogously to one in the paragraph before Corollary~\ref{cor:fund_gp_grass}.
Then
 $$
[\gamma]\in \pi_1\left(\P H^0(Y,\cL^{\otimes d})\setminus \Delta\right) \quad \text{is\ an\ element\ of\ infinite\ order.}
$$
\end{corollary}

We prove these statements 
in Section~\ref{sec:proofs_spheres_divs}. 
We have earlier explained the plan of proof of Theorem~\ref{th:twist_inf_grass};
actually we follow this plan to prove the general Theorem~\ref{th:twist_inf_general} first, and then derive
Theorem~\ref{th:twist_inf_grass} from it.

\subsection{Equivariant transversality approaches}
This  subsection is not used in the rest of the paper.
Computations of Floer cohomology in the presence of a symplectic involution were discussed by Khovanov and Seidel \cite{KhoSe02}, and Seidel and Smith \cite{SeSm10}. Both papers imposed restrictive conditions on the involution which allow one to choose a regular equivariant almost complex structure for computing Floer cohomology.

In \cite{SeSm10}, it is proved that
$$
\dim HF^*(L_1,L_2;\Z/2)\ge \dim HF^*(L_1^\ii, L_2^\ii;\Z/2)
$$
when there exists a stable normal trivialisation of the normal bundle to $X^\ii$ respecting the $L_i$.
In particular, the Chern classes of this normal bundle should vanish. 
The right-hand side is Floer cohomology inside $X^\ii$, where $L_i^\ii$ are the fixed loci of $L_i$ and $X^\ii$ is the fixed locus of $X$. Sometimes the right-hand side is easier
to compute than the left-hand side (e.g.~when all intersection points $L_1^\ii\cap L_2^\ii$ have the same sign).
However, the condition on the normal bundle makes this estimate inapplicable
to divisors in $Gr(k,n)$.

In a very special case,  \cite{KhoSe02} proves that
$$
\dim HF^*(L_1,L_2;\Z/2)=| L_1^\ii\cap L_2^\ii|
$$
where the right hand side is the unsigned count of intersection points.
The assumption is, roughly, that the fixed locus $X^\ii$ has real dimension $2$
and $L_1^\ii,L_2^\ii\subset X^\ii$ are curves having minimal intersection in their homotopy class.
One could prove a $\C$-version of this equality if the $L_i$ admit $\ii$-equivariant Pin strictures,
and apply it 
 to divisors in $\P^{n-1}=Gr(1,n)$, i.e.~projective hypersurfaces (thus giving an alternative proof of Theorem~\ref{th:twist_inf_grass} in this case).
However, it cannot be applied to divisors in general Grassmannians. When $k>2$, $Gr(k,n)$ has no holomorphic involution
with a  connected component of complex dimension 2;
this is easy to check
because all holomorphic automorphisms $Gr(k,n)$ come from linear ones on $\C^n$,
with a single exception when $n=2k$ \cite[Theorem~1.1 (Chow)]{Cow89}.

\subsection*{Acknowledgements}
The author is indebted to his supervisor Ivan~Smith for regular discussions and many helpful suggestions that shaped this paper. 
As already mentioned, one of the main results of this paper (Theorem~\ref{th:ell_rel}) was generously proposed  by Paul~Seidel. The author is also grateful to Baptiste~Chantraine, Jonny~Evans, Yank\i~Lekili, Andreas~Ott, Dmitri~Panov and Oscar~Randal-Williams
for useful discussions and comments, and the anonymous referee for suggesting valuable improvements to the exposition.

The author is funded by the Cambridge Commonwealth, European and~International Trust, and acknowledges travel funds from King's College, Cambridge.
\end{section}


\begin{section}{The elliptic relation}
\label{sec:ell_rel}
This section proves the elliptic relation for symplectomorphisms (together with its corollary, Proposition~\ref{prop:hf_bound_fixp_notrans}) and sketches a proof of the Lagrangian elliptic relation. 

\subsection{Floer cohomology and continuation maps}
\label{subsec:flo_hom}

\begin{definition}[The $W^+$ condition]
\label{def:weak_monot}
A symplectic manifold $(X,\omega)$ of dimension $2n$ satisfies the $W^+$ condition  \cite{Sei97}, if for every $A\in \pi_2(X)$
$$
2-n\le c_1(A) \le -1 \quad \Longrightarrow \quad \omega(A)\le 0.
$$ 
\end{definition}

Let $(X,\omega)$ be a compact symplectic manifold satisfying the $W^+$ condition. Fix a symplectomorphism $f\co X\to X$.
In this subsection we recall the definition of Floer cohomology $HF^*(f)$; basic references are
\cite{MDSaBook, DoSa94, Sei97}.

Take a family of  $\omega$-tame almost complex structures $J_s$ on $X$, 
and a family of Hamiltonian functions $H_s\co X\to \R$, 
$s\in \R$. They must be $f$-periodic:
\begin{equation}
\label{eq:h_j_period_cyl}
H_s=H_{s+1}\circ f, \qquad J_s=f^* J_{s+1}. 
\end{equation}

By $X_{H_s}$ we denote the Hamiltonian vector field of $H_s$,
and by $\psi_s\co X\to X$ the Hamiltonian flow:
\begin{equation}
\label{eq:hs_flow}
d \psi_s/d s=X_{H_s}\circ \psi_s, \qquad \psi_0=\id. 
\end{equation}

The following equation on $u(s,t)\co \R^2\to X$ is called Floer's equation:
\begin{equation}
\label{eq:Floer_dif}
 \bd u/\bd t + J_s(u)(\bd u/\bd s-X_{H_s}(u))=0.
\end{equation}

This equation comes with the periodicity conditions
\begin{equation}
\label{eq:u_period_cyl}
u(s+1,t)=f(u(s,t)).
 \end{equation}

Denote 
\begin{equation}
\label{eq:f_H}
f_H\coloneqq \psi_1^{-1}\circ f \in \Symp(X).
\end{equation}
(The correct notation would be $f_{H_s}$, but we stick to $f_H$ for brevity). 
Suppose the fixed points of $f_H$ are isolated and non-degenerate (that is to say, for every $x\in \fix f_H$,
$\ker (\id-df_H(x))=0$). 
Then finite energy solutions to Floer's equation have the following convergence property.
There exist points $x,y$ 
such that
\begin{equation}
\label{eq:sol_asympt}
 \lim_{t\to -\infty}u(s,t)=\psi_s(x), \quad \lim_{t\to +\infty}u(s,t)=\psi_s(y), \quad x,y\in \fix f_H.
\end{equation}

For $x,y\in \fix f_H$, let $\M(x,y;J_s,H_s)$ be the moduli space of all solutions
to Floer's equation~(\ref{eq:Floer_dif}) with limits~(\ref{eq:sol_asympt}).
For regular $J_s,H_s$, the moduli space is a  manifold which is a disjoint union of
the $k$-dimensional pieces $\M^k(x,y;J_s,H_s)$. 
They can be oriented in a way consistent with gluings;
such orientations are called coherent \cite{FloHo93}.
There is an $\R$-action on $\M(x,y;J_s,H_s)$, and once a coherent orientation is fixed,
$\M^1(x,y;J_s,H_s)/\R$ is a set of signed points.

The Floer complex associated to $(f;J_s,H_s)$ is the $\Lambda$-vector space generated by points in $\fix f_H$:
$$
CF^*(f;J_s,H_s)\coloneqq \bigoplus_{x\in \fix f_H} \Lambda\langle x\rangle.
$$

The differential on $CF^*(f;J_s,H_s)$ is defined on a generator $x\in \fix f_H$ by:
\begin{equation}
\label{eq:floer_dif_sum}
\bd(x)=\ \ 
\sum_{
\mathclap{\begin{smallmatrix}
y\in \fix f_H \\
u\in \M^1(x,y;J_s,H_s)/\R
\end{smallmatrix}}
}\ \  \
\pm q^{\omega(u)}\cdot y.
\end{equation}

Here the signs are those of the points in $\M^1(x,y;J_s,H_s)/\R$, and
\begin{equation}
\label{eq:omega_area_cyl}
 \omega(u)=\int_{s\in [0,1]}\int_{t\in \R} u^* \omega\ dsdt.
\end{equation}

Suppose $J_s,H_s$ and $J_s',H_s'$ are two regular choices of almost complex structures and Hamiltonians
that satisfy the $f$-periodicity condition (\ref{eq:h_j_period_cyl}).
Choose a family of $\omega$-tame complex structures 
$J_{s,t}$ and Hamiltonians $H_{s,t}$, $s,t\in \R$, such that for each $t$, Condition~(\ref{eq:h_j_period_cyl}) is satisfied and
\begin{equation}
\label{eq:cont_htopy_support}
J_{s,t}\equiv J_s', \ H_{s,t}\equiv H_{s}'\ \mathrm{ for }\ t\ \mathrm{ near}\ -\infty, \quad
J_{s,t}\equiv J_s, \ H_{s,t}\equiv H_{s}\ \mathrm{ for }\ t\ \mathrm{ near}\ +\infty. 
\end{equation}
We call $J_{s,t},H_{s,t}$ a homotopy from $J_s',H_s'$ to $J_s,H_s$.

Define
$\M(x,y;J_{s,t},H_{s,t})$ to be the set of solutions to Floer's continuation equation
\begin{equation}
\label{eq:Floer_cont}
 \bd u/\bd t + J_{s,t}(u)(\bd u/\bd s-X_{H_{s,t}}(u))=0
\end{equation}
with periodicity condition~(\ref{eq:u_period_cyl}) and asymptotic conditions:
\begin{equation}
\label{eq:sol_asympt_cont}
 \lim_{t\to -\infty}u(s,t)=\psi_s(x), \quad \lim_{t\to +\infty}u(s,t)=\psi_s(y), \quad x\in \fix f_{H'},\ y\in \fix f_H.
\end{equation}

If $J_{s,t},H_{s,t}$ are regular, $\M(x,y;J_{s,t},H_{s,t})$ is a manifold.
Let $\M^0(x,y;J_{s,t},H_{s,t})$ be its 0-dimensional component, which is  a collection of
signed points once coherent orientations (consistent with those for $J_s,H_s$ and $J_s',H_s'$) are fixed. Define the continuation map
$C_{J_{s,t},H_{s,t}}\co CF^*(f;J_s',H_s')\to CF^*(f;J_s,H_s)$ by
\begin{equation}
\label{eq:cont_map}
C_{J_{s,t},H_{s,t}}(x)=\ \ 
\sum_{
\mathclap{\begin{smallmatrix}
y\in \fix f_{H} \\
u\in \M^0(x,y;J_{s,t},H_{s,t})
\end{smallmatrix}}
}\ \ \ 
\pm q^{\omega(u)}\cdot y. 
\end{equation}
Here $x\in \fix f_{H'}$. For regular $J_{s,t},H_{s,t}$, it is a chain map inducing an isomorphism on cohomology.
So one can actually identify the homologies $HF^*(f;J_s,H_s)$ for all generic $J_s,H_s$
to get a single space $HF^*(f)$. It is called Floer cohomology of $f$. It is a $\Z/2$-graded  vector space over $\Lambda$; we will recall  the grading later.

\subsection{Commuting symplectomorphisms induce actions on Floer cohomology}
\label{subsec:comm_symp}

As before, let $X$ be a compact symplectic manifold satisfying the $W^+$ condition. Let $f,g\co X\to X$ be two commuting symplectomorphisms;
we will now define an automorphism $g\hf\co  HF^*(f)\to HF^*(f)$.

Pick generic $J_s,H_s$ that satisfy (\ref{eq:h_j_period_cyl}) to get the complex $CF^*(f;J_s,H_s)$.
Denote
\begin{equation}
\label{eq:j_h_pull_g}
J_s'\coloneqq g^*J_s,\quad H_s'\coloneqq  H_s \circ g.
\end{equation}
This gives us another complex
$CF^*(f;J_s',H_s')$.
Note that
$g\circ \psi_1=\psi_1'$.
Let us check that $ f_H=f_{H'}\circ g$:
$$
f_{H'}\circ g(x)=(\psi_1')^{-1}fg(x)=(\psi_1')^{-1}gf(x)=\psi_1^{-1}f(x)=f_H(x).
$$
Consequently,
 $g$ induces a bijection $\fix f_H\to \fix f_{H'}$.
 Extend it by $\Lambda$-linearity to
$$
g\ps\co CF^*(f;J_s,H_s)\to  CF^*(f;J_s',H_s').
$$
Similarly, the composition map
$u\mapsto g\circ u$ is an isomorphism 
$$\M(x,y;J_s,H_s)\stackrel{\cong}\longrightarrow \M(g(x),g(y);J_s',H_s').$$ 
So $g\ps$ is tautologically a chain map inducing an isomorphism on cohomology.

Now fix a homotopy $J_{s,t},H_{s,t}$ from $J_s',H_s'$ to $J_s,H_s$ as in (\ref{eq:cont_htopy_support}).
Consider the composition
$$
CF^*(f;J_s,H_s)\xrightarrow{g\ps} CF^*(f; J_s',H_s') \xrightarrow{C_{J_{s,t},H_{s,t}}} CF^*(f;J_s,H_s).
$$

\begin{definition}[Action on Floer cohomology]
\label{def:action_hf}
We define $g\hf\co HF^*(f;J_{s,t},H_{s,t})\to HF^*(f;J_{s,t},H_{s,t})$
to be the map induced by the composition of chain maps $C_{J_{s,t},H_{s,t}}\circ g\ps$. We will frequently suppress 
the choice of $J_s,H_s$  and simply write
$g\hf\co HF^*(f)\to HF^*(f)$. Also, we will sometimes denote the chain-level map by the same symbol,
$g\hf=C_{J_{s,t},H_{s,t}}\circ g\ps$.

As a part of this definition, the signs in  formula (\ref{eq:cont_map}) for $C_{J_{s,t},H_{s,t}}$
must come from a coherent orientation as explained in Subsection~\ref{subsec:coherent_orient_cont} below.
In particular, for any $x\in \fix f_H$, the sign of an element $u\in \M^0(g(x),x;J_{s,t},H_{s,t})$  is canonical, see Definition~\ref{def:sign_sec_cyl},
and denoted by $\sgn(u)$.
\end{definition}

\begin{remark}
On the level of cohomology, 
$g\hf$ does not depend on the chosen homotopy
$J_{s,t},H_{s,t}$; this follows from the fact that the continuation map $C_{J_{s,t},H_{s,t}}$ does not depend 
on the choice of homotopy,  see e.g.~\cite[Section~12.1]{MDSaBook}.
\end{remark}

\begin{remark}[An analogue in Morse cohomology]
\label{rem:morse_cont}
A similar construction is known in Morse cohomology \cite[4.2.2]{SchwBook}. Suppose $H\co X\to \R$ is a Morse-Smale function on a Riemannian manifold $(X,g)$, and $f\co X\to X$ is a diffeomorphism.
Let $C^*(H)$ be the Morse complex of $X$ generated by points in $\Crit(H)$.
Pick homotopies $H_t$ from $H\circ f$ to $H$, and $g_t$ from $f^*g$ to $g$, and
define $f^*\co C^*(H)\to C^*(H)$ as follows. Take $x,y\in \Crit(H)$ and
let the coefficient of  $f^*(x)$ on $y$ be
the signed count of flowlines of the gradient $\nabla_{g_t}H_t$ going from $f(x)$ to $y$.
The chain map $f^*$ induces an automorphism of $H^*(X)$ known from elementary topology.

In particular, let us note for future use that the Lefschetz number $L(f)$ can be computed as the sum, over $x\in \Crit(H)$, of $\nabla_{g_t}H_t$-flowlines going from $f(x)$ to $x$, counted with signs.
\end{remark}

\begin{remark}[Relation to  Seidel elements]
If $g$ is Hamiltonian isotopic to $f$ through symplectomorphisms commuting with $f$,
then one can show $g\hf\co  HF^*(f)\to HF^*(f)$ is the identity. If $g$ is just Hamiltonian isotopic to $f$,  $g\hf$ need not be the identity, but can be understood as follows.
Take a homotopy $g_t$, $g_0=g,g_1=f$. The path $\gamma_t\coloneqq g_t^{-1}f g_t$ is actually a loop in $Symp(X)$: $\gamma(0)=\gamma(1)=f$ because $g^{-1}fg=f$. To this path one associates its Seidel element, $S(\gamma)\in QH^*(M;\Lambda)$ \cite{Sei97}.
Let $*$ be the quantum multiplication $QH^*(M;\Lambda)\otimes HF^*(f)\to HF^*(f)$. One can check
that
$
g\hf(x)= S(\gamma)*x
$
for any $x\in HF^*(f)$. We will not use this observation, so we omit its proof.
\end{remark}

\subsection{Iterations}
If $f,g$ commute then $f,g^k$ also commute for any iteration $g^k$.

\begin{lemma}
\label{lem:action_powers}
The following two automorphisms of $HF^*(f)$ are equal:
$$
(g\hf)^k=(g^k)\hf.
$$ 
\end{lemma}

\begin{proof}
We prove the case $k=2$; the general case is analogous. Take $J_s,H_s$ as in~(\ref{eq:h_j_period_cyl}), $J_s',H_s'$ pulled by $g$ as in (\ref{eq:j_h_pull_g}) 
 and the homotopy $J_{s,t},H_{s,t}$ as in (\ref{eq:cont_htopy_support}). Denote 
$$
J_s^{''}=g^*J_s'=(g^2)^*J_s, \quad H_s^{''}= H_s'\circ g= H_s\circ g^2.
$$
Compare the two compositions given below. The first one induces $(g\hf)^2$ on the homological level:
\begin{multline*}
CF^*(f;J_s,H_s)\xrightarrow{g\ps} \\
CF^*(f; J_s',H_s') \xrightarrow{C_{J_{s,t},H_{s,t}}} CF^*(f;J_s,H_s)
\xrightarrow{g\ps} CF^*(f; J_s',H_s') \\
\xrightarrow{C_{J_{s,t},H_{s,t}}} CF^*(f;J_s,H_s)
\end{multline*}
The second composition gives $(g^2)\hf$, by a gluing theorem for continuation maps:
\begin{multline*}
CF^*(f;J_s,H_s)\xrightarrow{g\ps} \\
CF^*(f; J_s',H_s') \xrightarrow{g\ps} CF^*(f;J_s'',H_s'')
\xrightarrow{C_{J_{s,t}',H'_{s,t}}} CF^*(f; J_s',H_s') \\
\xrightarrow{C_{J_{s,t},H_{s,t}}} CF^*(f;J_s,H_s)
\end{multline*}
By definition of $J_{s,t}',H_{s,t}'$ (\ref{eq:j_h_pull_g}), $g$ maps Floer solutions (\ref{eq:Floer_cont}) in $\M(x,y;J_{s,t},H_{s,t})$ to those in $\M(g(x),g(y);J_{s,t}',H_{s,t}')$. This means
$$
C_{J_{s,t}',H_{s,t}'}\circ g\ps=g\ps \circ  C_{J_{s,t},H_{s,t}}.
$$
This proves Lemma~\ref{lem:action_powers}.
\end{proof}


\subsection{Supertrace}
We continue to use  notation from Subsection~\ref{subsec:flo_hom}.

\begin{definition}[Grading on Floer's complex]
\label{def:floer_grading}
Let $x\in \fix f_H$.
We say $\deg x=0$ if $\mathrm{sign} \det(\id-df_H(x))>0$
and $\deg x=1$ otherwise.
\end{definition}

This makes $CF^*(f;J_s,H_s)$ a $\Z_2$-graded vector space over $\Lambda$. Floer's differential has degree $1$,
so the cohomology is also $\Z_2$-graded: $HF^*(f)=HF^0(f)\oplus HF^1(f)$.

\begin{definition}[Supertrace]
\label{def:str}
 Let $V=V^0\oplus V^1$ be a $\Z_2$-graded vector space and $\phi\co V\to V$
 an automorphism of zero degree, i.e.~ $\phi(V^0)\subset V^0$, $\phi(V^1)\subset V^1$.
Then $STr (\phi)\coloneqq  Tr (\phi|_{V^0})-Tr(\phi|_{V^1})$.
\end{definition}

The automorphism $g\hf$ from Definition~\ref{def:action_hf} has zero degree,
so it has well-defined supertrace which is an element of $\Lambda$.
Supertraces can be computed on the chain level, since all our chain complexes are finite-dimensional.
Therefore the following is just a restatement of definitions.

\begin{lemma}
\label{lem:trace_on_chains}
Let $X$ be a symplectic manifold satisfying the $W^+$ condition and $f,g\co X\to X$ be two commuting symplectomorphisms.
Take $J_s',H_s'$ as in (\ref{eq:j_h_pull_g}) and a homotopy $J_{s,t},H_{s,t}$ from $J_s',H_s'$ to $J_s,H_s$ as in (\ref{eq:cont_htopy_support}).
Then 
$$
STr(g\hf\co HF^*(f)\to HF^*(f))=\ \ \sum_{
\mathclap{\begin{smallmatrix}
 x\in \fix f_H,\\
 u\in \M^0(g(x),x;J_{s,t},H_{s,t})
\end{smallmatrix}}
}\ \ \ 
(-1)^{\deg x}\cdot \sgn(u)\cdot q^{\omega(u)}
$$
where $\sgn(u)=\pm 1$ is defined in Definition~\ref{def:sign_sec_cyl}.
\end{lemma}

\begin{proof}
Pick a generator  $x\in \fix f_H$ of $CF^*(f;J_s,H_s)$.
Rewriting the definition of $g\hf$  we get:
$$
g\hf(x)=\ \ \sum_{\mathclap{{u\in \M^0(g(x),y;J_{s,t},H_{s,t})}}}\ \ \ \pm q^{\omega(u)}\cdot y.
$$
When we put $x=y$, the sign $\pm$ is substituted by
$\sgn(u)$ according to Definition~\ref{def:action_hf}.
\end{proof}



\subsection{Holomorphic sections}
\label{subsec:hol_sect}
It is useful to reformulate the definition of Floer cohomology using holomorphic sections
as in e.g.~\cite{SeiL4D}.
If $f\co X\to X$ is a symplectomorphism, consider the mapping cylinder
\begin{equation}
\label{eq:E_f}
E_f\coloneqq \frac{X\times \R^2_{s,t}}{(x,s,t)\sim (f(x),s+1,t)}.
\end{equation}
There is a closed 2-form $\omega_{E_f}$ on $E_f$
which comes from $\omega\oplus 0$ on $X\times\R^2$, and
a natural fibration $p\co E_f\to S^1\times \R$ whose fibres are symplectomorphic to $X$.

The $f$-periodicity condition (\ref{eq:u_period_cyl}) on $u\co \R^2\to X$ means that it can be seen as a section
 $u\co S^1\times \R\to E_f$.
 Floer's equation itself (\ref{eq:Floer_dif}) is equivalent to $u$ being a holomorphic section
with respect to the standard complex structure $j_{S^1\times\R}$ on $S^1\times \R$ and an almost complex structure $\tilde J$ on
$E_f$. In other words, Floer's equation (\ref{eq:Floer_dif}) becomes:
\begin{equation}
\label{eq:holoc_sect}
du+ \tilde J\circ du\circ j_{S^1\times \R}=0.
\end{equation}
The almost complex structure $\tilde J\coloneqq \tilde J(J_s,H_s)$
is determined by $J_s$ and $H_s$,  see e.g.~\cite[Section~8.1]{MDSaBook}. 
Analogously, if $J_{t,s},H_{t,s}$ is a continuation homotopy (\ref{eq:cont_htopy_support}),
the moduli space $\M(x,y;J_{s,t},H_{s,t})$
consists of sections $u\co S^1\times \R\to E_f$ that are holomorphic with respect to 
$j_{S^1\times \R}$
and an almost complex structure $\tilde J(J_{s,t},H_{s,t})$ on $E_f$.


\subsection{Asymptotic linearised Floer's equation}
\label{subsec:asymp_lin_floer}
Let $E_f$ be as in (\ref{eq:E_f}).
We denote
$$T^vE_f=\ker dp$$ 
the vertical tangent bundle of $E_f$.
The almost complex structures $J_s$ turn $T^vE_f$ into a complex vector bundle.
Take a solution $u(s,t)$ to Floer's equation, $u\in \M(x,y;J_s,H_s)$. We regard it as a section  $u(s,t)\co  S^1\times \R\to E_f$  as explained above. 
The pullback $u^* T^vE_f$ is a complex vector bundle over $S^1\times \R$.
By linearising Floer's equation  (\ref{eq:holoc_sect}), one gets a map
\begin{equation}
\label{eq:linear_floer}
D_u\co H^{1,p}(u^* T^{v}E_f) \to L^p(\Omega^{0,1}(u^*T^{v}E_f)).
\end{equation}
Here $\Omega^{0,1}(u^*T^{v}E_f)$ consists of bundle maps $T(S^1\times \R)\to u^*T^vE_f$ which are
complex-antilinear with respect to $\tilde J$ and the standard complex structure on $S^1\times \R$.

We know from (\ref{eq:sol_asympt}) that $u$ extends to $S^1\times \{\pm \infty\}$:
$u(s,-\infty)=\psi_s(x)$ where $\psi_s$ is the flow (\ref{eq:hs_flow}) of $X_{H_s}$.
(The same is true of $t\to +\infty$ and the point $y$. We will now speak of $t\to -\infty$ only.)
Choose a complex trivialisation
\begin{equation}
\label{eq:phi_x_triv}
\Phi_{x}\co u^*T^vE_f|_{S^1\times \{-\infty\}}\to S^1\times \R^{2n}.
\end{equation}
We choose a single trivialisation for each point $x$; this is possible because $u(s,-\infty)=\psi_s(x)$.
The operator $D_u$ is asymptotic, as $t\to -\infty$, to the  operator
\begin{equation}
\label{eq:L_x_cyl}
L_{A(s)}=\bd/\bd t+ J_0\bd/\bd s+A(s)\co  H^{1,p}(S^1\times \R,\R^{2n})\to L^p(S^1\times \R,\R^{2n}).
\end{equation}

Here $J_0$ is the standard complex structure on $\R^{2n}$,
 and $A(s)$ is a map $S^1\to Hom(\R^{2n},\R^{2n})$ taking values in symmetric matrices. It is known that $A(s)$ is determined by $J_s,H_s$, the point $x$ and the chosen trivialisation $\Phi(x)$. It does not depend on $u$ as long as the $t\to -\infty$ asymptotic of $u$ stays fixed.
 A reference for these facts is (among others) the thesis of Schwarz  \cite[Definition~3.1.6, Theorem~3.1.31]{SchwThesis}.
Although  that thesis only considers the case $f=\id$,
the proofs of the results we use are valid for any $f$, as these are general results about certain Fredholm operators on bundles over $S^1$ and $S^1\times \R$.

\begin{lemma}[{\cite[proof of Lemma~3.1.33]{SchwThesis}}]
\label{lem:asymp_oper_cyl_solutions}
Consider the operator
$$
J_0\bd/\bd s + A(s)\co C^\infty(S^1,\R^{2n})\to C^\infty(S^1,\R^{2n}).
$$
 There is a family of linear maps 
$
\Psi(s)\co  [0,1]\to Sp(\R^{2n})
$
such that 
\begin{equation}
\label{eq:psi}
\left(J_0\bd/\bd s +A(s)\right) \Psi(s)=0, \quad \Psi(0)=\id
\end{equation}
 and  $\Psi(1)\co \R^{2n}\to\R^{2n}$
coincides, under the trivialisation $\Phi_x$ (\ref{eq:phi_x_triv}), with the differential $df_H(x)$.\qed
\end{lemma}

\begin{remark}
We identify $S^1=\R/\Z$ so points of the circle $s=0$ and $s=1$ are the same.
The statement about $\Psi(1)$ in the lemma above  makes sense because 
$u(0,-\infty)=x$ for some $x\in \fix f_H$, see (\ref{eq:sol_asympt}) and (\ref{eq:hs_flow}).
So $df_H(x)$ acts on $T_x X=u^*T^vE_f|_{(0,-\infty)}$.
The trivialisation (\ref{eq:phi_x_triv}) identifies this space with $\R^{2n}$.
\end{remark}

\begin{remark}
 Given $\Psi(s): [0,1]\to Sp(\R^{2n})$, by solving (\ref{eq:psi}) we get
\begin{equation}
\label{eq:psi_to_a}
A(s)=-J_0(\bd/\bd s\Psi(s))\Psi(s)^{-1}
 \end{equation}
with symmetric $A(s)\co [0,1]\to Hom(\R^{2n},\R^{2n})$. Conversely, we can go from $A(s)$ to $\Psi(s)$ by solving (\ref{eq:psi}) as an ODE.
\end{remark}

\begin{remark} For reader's convenience, we include a correspondence between our notation and that of Schwarz~\cite{SchwThesis}, our notation being on the left in each pair: $s\leftrightarrow t$, $t\leftrightarrow s$, $A(s)\leftrightarrow S^\infty(t)$, $\Psi(s)\leftrightarrow \Psi(t)$, and $D_u$ in our notation corresponds to either $D_u$ or $DF_h$, the latter being the linearisation of Floer's equation at an $h$ which is not necessarily a solution. Equation (\ref{eq:psi}) is \cite[(3.23)]{SchwThesis}.
\end{remark}

\subsection{An index problem on the torus}
\label{subsec:index_prob_torus}
The operator $L_{A(s)}$ (\ref{eq:L_x_cyl}) is Fredholm if and only if $\det (\id-\Psi(1))=\det(\id-df_H(x))$ is non-zero.
Now, for later use, consider variables $(s,t)$ belonging to the torus $S^1\times S^1$
instead of the cylinder $S^1\times \R$. The same formula $(\ref{eq:L_x_cyl})$
gives the operator
$$L_{A(s)}=\bd/\bd t+J_0\bd/\bd s+A(s)\co C^\infty (S^1\times S^1,\R^{2n})\to C^\infty (S^1\times S^1,\R^{2n})$$
which is now Fredholm of zero index for any family of symmetric matrices $A(s)\co S^1\to Hom(\R^{2n},\R^{2n})$.
For the remainder of this subsection, $L_{A(s)}$ denotes the operator on $S^1\times S^1$ and not on the cylinder.

\begin{lemma}
\label{lem:ker_asymp_oper_torus}
 Let $A(s)\co S^1\to Hom(\R^{2n},\R^{2n})$ be a family of symmetric matrices. Suppose $A(s)$ and $\Psi(s)$ satisfy (\ref{eq:psi}).
Then $\dim \ker L_{A(s)}=\dim \ker (\id-\Psi(1))$.
\end{lemma}
\begin{proof}
 Any $\xi(s,t)\in \ker L_{A(s)}$ must be independent of $t$, see
\cite[Proof of Lemma~3.1.33]{SchwThesis}, so we write $\xi(s,t)\equiv\xi(s)$. The equation on $\xi(s)$ becomes $(J_0\bd/\bd s+A(s))\xi(s)=0$.
This is an ODE whose solutions are of form $\xi(s)=\Psi(s)v$ for some $v\in \R^{2n}$ by (\ref{eq:psi}).
There are no other solutions by the uniqueness theorem for ODEs, as $v\in \R^{2n}$ sweep out all initial conditions.
In addition, our solutions must close up on the circle, meaning $\xi(1)=\xi(0)$, which forces $\Psi(1)v=v$.
\end{proof}

Let $A_0(s),A_1(s)\co S^1\to Hom(\R^{2n},\R^{2n})$ be two  families of symmetric matrices with $L_{A_0(s)}$, $L_{A_1(s)}$ injective. Choose a generic smooth homotopy $A_\tau(s)$ between them, $\tau\in [0,1]$.
Define $\sgn(L_{A_0(s)},L_{A_1(s)})=(-1)^\epsilon$  where
\begin{equation}
\label{eq:sign_count_torus}
\epsilon =\sum_{\tau\in[0,1]} \dim_\R \ker L_{A_\tau(s)}.
\end{equation}
This sum contains a finite number of non-zero terms as $L_{A_\tau(s)}$ are generically injective, and does not depend modulo 2 on the chosen homotopy.

\begin{lemma}
 \label{lem:sign_count_torus_compute}
For $A_0(s)$, $A_1(s)$ as above and $\Psi_0(s),\Psi_1(s)$ satisfying (\ref{eq:psi}), 
we have $\sgn(L_{A_0(s)},L_{A_1(s)})=\sgn \det (\id-\Psi_0(1))\cdot \sgn \det (\id-\Psi_1(1))$.
\end{lemma}

\begin{proof}
For $i=0,1$ denote $\tilde \Psi_i(s)=e^{s\log \Psi_i(1)}$, $s\in [0;1]$, so that $\tilde \Psi_i(0)=\Psi_i(0)=\id$ and $\tilde \Psi_i(1)=\Psi_i(1)$.
Let us compute $\tilde A_i(s)$ from $\tilde \Psi_i(s)$ using (\ref{eq:psi_to_a}):
$\tilde A_i(s)=-J_0(\bd/\bd s\tilde \Psi_i(s))\tilde \Psi_i(s)^{-1}=-J_0\log\Psi_i(1)$. We see it is a constant $s$-independent symmetric matrix $\tilde A_i(s)\equiv \tilde A_i$. 
Our first claim is that 
\begin{equation}
\label{eq:index_sign_plus_1}
\sgn(L_{A_i(s)},L_{\tilde A_i})=+1.
\end{equation}
Indeed,  choose the homotopy
$$
(\Psi_i)_\tau(s)=e^{\tau s\log \Psi_i(1)}e^{(1-\tau)\log \Psi_i(s)}
$$
from $\Psi_i(s)$ to $\tilde \Psi_i(s)$, where $\tau \in[0;1]$, and observe this homotopy has fixed endpoints: we have $(\Psi_i)_\tau(0)=\Psi_i(0)$ for each $\tau$, and also $(\Psi_i)_\tau(1)=\Psi_i(1)$. Passing from $(\Psi_i)_\tau(s)$ to
$(A_i)_\tau(s)$ by formula (\ref{eq:psi})
we get the linear homotopy
 $$(A_i)_\tau(s)=\tau A_i(s)+(1-\tau) \tilde A_i$$ from $A_i(s)$ to $\tilde A_i$.
The corresponding  operator $L_{(A_i)_\tau(s)}$ is injective for all $\tau$ by Lemma~\ref{lem:ker_asymp_oper_torus}
because we are given $\ker (\id-\Psi_i(1))=0$.
This implies that $$\sgn(L_{A_i(s)},L_{\tilde A_i})=+1,$$ as desired.

Let us compute $\sgn(L_{\tilde A_0},L_{\tilde A_1})$ for two constant  matrices  $\tilde A_i(s)\equiv \tilde A_i$, $i=0,1$.
By linear algebra, one can find a smooth path of matrices $\tilde A_\tau$ from $\tilde A_0$ to $\tilde A_1$ such that
$(-1)^{\sum_{\tau}\dim\ker \tilde A_\tau}=\sgn\det \tilde A_0\cdot \sgn \det \tilde A_1$.
We will now show that $\dim \ker \tilde A_\tau=\dim \ker L_{\tilde A_\tau}$ for each $\tau$,
and this will immediately imply that
\begin{equation}
\label{eq:index_sign_const_matr}
\sgn(L_{\tilde A_0},L_{\tilde A_1})=\sgn\det \tilde A_0\cdot\sgn \det \tilde A_1.
\end{equation}
For the rest of the paragraph, redenote $\tilde A_\tau$ (for some fixed $\tau$) by $A$: this is an arbitrary symmetric matrix, and if we consider it as an $s$-independent family and solve (\ref{eq:psi_to_a}) with respect to $\Psi$, we get $\Psi(1)=e^{-J_0A}$. By Lemma~\ref{lem:ker_asymp_oper_torus}, we have $\ker L_A=\ker (\id-\Psi(1))=\ker (\id-e^{-J_0A})$. The latter equals $\ker A$, as seen by bringing $A$ to the Jordan normal form.
So (\ref{eq:index_sign_const_matr}) is now justified.
Combining all above, we get
\begin{align*}
\sgn(L_{A_0(s)},L_{A_1(s)})
&=
\sgn(L_{A_0(s)},L_{\tilde A_0})\cdot\sgn(L_{\tilde A_0},L_{\tilde A_1})\cdot\sgn(L_{\tilde A_1},L_{A_1(s)})\\
 &
=\sgn(\det \tilde A_0)\cdot\sgn (\det \tilde  A_1). 
\end{align*}
The first equality is true because we can regard the concatenation of three homotopies between the operators appearing in the middle expression as a single homotopy between the eventual endpoints $L_{A_0(s)}$ and $L_{A_1(s)}$; the second equality follows from 
(\ref{eq:index_sign_plus_1}) and (\ref{eq:index_sign_const_matr}).
Finally, recall  $\tilde A_i=-J_0\log \Psi_i(1)$
and observe that 
$\sgn\det \log\Psi_i(1)=\sgn\det(\id-\Psi_i(1))$.
This completes the proof.
\end{proof}

\subsection{Signs for the action on Floer cohomology}
\label{subsec:coherent_orient_cont}
Let $f,g$ be two commuting symplectomorphisms.
We will now complete Definition~\ref{def:action_hf} of the action $$g\hf\co  HF^*(f)\to HF^*(f)$$ 
by specifying the signs appearing there.

Pick regular $J_s,H_s$ to define
 Floer's complex $CF^*(f;J_s,H_s)$.
For each $x\in \fix f_H$, pick a trivialisation $\Phi_x$ (\ref{eq:phi_x_triv}).
Then for each $x$, we get a unique asymptotic linearised operator
$L_{A_x(s)}$ (\ref{eq:L_x_cyl}).

Let $J_s',H_s'$ be pulled by $g$ (\ref{eq:j_h_pull_g}) and $J_{s,t},H_{s,t}$ be a homotopy  (\ref{eq:cont_htopy_support}).
Let $u\in\M^0(g(x),y;J_{s,t},H_{s,t})$ be
a solution to Floer's continuation equation, where $x,y\in \fix f_H$ so that $g(x)\in\fix f_{H'}$.
Consider the linearisation $D_u$ of Floer's continuation equation at $u$;
its properties are similar to those discussed in Subsection~\ref{subsec:asymp_lin_floer}.
As $t\to +\infty$, $D_u$ is asymptotic to $L_{A_y(s)}$ because
for $t$ close to $+\infty$, $J_{s,t},H_{s,t}$ are equal to $J_s,H_s$.
On the other hand, as $t\to-\infty$, we can write down $D_u$
in the $g$-induced trivialisation $\Phi_x\circ dg$ of $u^*TE_f|_{u(-\infty,s)}$.
We claim that $D_u$ is asymptotic, as $t\to -\infty$, to $L_{A_x(s)}$.
Indeed, the asymptotic operator is determined
by the following data: the fixed point $g(x)$, the chosen trivialisation $\Phi_x\circ dg$,
and $J_{s,t},H_{s,t}$ which equal $g^*J_s,H_s\circ g$ for $t$ close to $-\infty$.
We see all of this data is pulled by $g$ from the data $x$, $\Phi_x$, $J_s$, $H_s$
which defines the asymptotic linearised operator $A_x(s)$. Clearly, pullback by $g$
does not change the linearised operator at all, so $D_u$ is asymptotic to $L_{A_x(s)}$ as $t\to-\infty$.

The outcome is that the set $\{L_{A_x(s)}\}_{x\in\fix f_H}$ of asymptotic operators
to $D_u$ for $u\in \M(x,y;J_s,H_s)$ (these are solutions to Floer's equations for the differential on $CF^*(f)$, without the second symplectomorphism $g$  involved)
is identical to the set of asymptotic operators
to $D_u$ for $u\in\M(g(x),y;J_{s,t},H_{s,t})$ (these are solutions to Floer's continuation equation),
provided we use the described trivialisations.

Consequently, the usual definition of coherent orientations \cite{FloHo93} 
 on $\M(x,y;J_s,H_s)$
can be applied without any change to orient $\M(g(x),y;J_{s,t},H_{s,t})$,
$x,y\in\fix f_H$.
In Definition~\ref{def:action_hf}, we pick such a coherent orientation on
 $\M(g(x),y;J_{s,t},H_{s,t})$.
Instead of repeating the complete definition of coherent orientations, we only recall a piece relevant
to the signs appearing in Lemma~\ref{lem:trace_on_chains} regarding
the supertrace of $g\hf$.

Coherent orientations are not unique, but the sign any coherent orientation associates to a point 
$u\in \M^0(g(x),x;J_{s,t},H_{s,t})$, $x\in \fix f_H$,
is canonical. We explain its definition
following \cite{FloHo93} and \cite[Appendix A]{MDSaBook}.
As we have seen, $D_u$ is asymptotic as $t\to\pm \infty$ to the same operator
$$
L_{A(s)}=\bd/\bd t+J_0\bd/\bd s+A(s),
$$
where $A(s)=A_x(s)$ in notation of the previous paragraphs.
Choose a generic homotopy $L_\tau$ 
from $D_u$ to $L_{A(s)}$, $\tau\in[0,1]$, such that $L_\tau$ are Fredholm operators
which stay asymptotic to $L_{A(s)}$ as $t\to\pm\infty$.

\begin{definition}
\label{def:sign_sec_cyl}
For $u\in \M^0(g(x),x;J_{s,t},H_{s,t})$, $x\in \fix f_H$, define $\sgn(u)=(-1)^\epsilon$ where
$$
\epsilon=\sum_{\tau\in[0,1]}\dim_\R\ker L_\tau.
$$
\end{definition}

Because the operators $L_\tau$ have zero index, the sum is well-defined and does not depend modulo 2 on the chosen path.
Let us repeat that, as part of Definition~\ref{def:action_hf}, these signs appear in Lemma~\ref{lem:trace_on_chains}.

\subsection{Holomorphic sections over the torus}
\label{subsec:fib_torus}
Subsection~\ref{subsec:hol_sect} explained that solutions to Floer's equation are
holomorphic sections of a fibration $E_f\to S^1\times \R$,
whose monodromy around $S^1$ equals $f$.
Now, let $f,g$ be two commuting symplectomorphisms of $X$. In this subsection
we define a fibration $p\co E_{f,g}^{1,R}\to T^{1,R}$
over a 2-torus $T^{1,R}$.
The monodromies of this fibration equal $f$ and $g$ around the two basis loops of the torus.
After that we recall how to count its holomorphic sections, see \cite{MDSaBook} for details. 
We start by defining the torus
$$
T^{1,R}\coloneqq \frac{[0,1]\times [-R,R]}{(s,\{-R\})\sim (s,\{R\}),\ (\{0\},t)\sim (\{1\},t)}
$$
and equipping  $T^{1,R}$ with the complex structure $j^{1,R}$
which comes from the standard one on $[0,1]\times \sqrt{-1}[-R,R]\subset \C$.
Define
$$
\label{eq:E_f_g}
E_{f,g}^{1,R}\coloneqq \frac{X\times [0,1]\times [-R,R]}{(x,s,\{-R\})\sim (g(x),s,\{R\}),\ (x,\{0\},t)\sim (f(x),\{1\},t)}
$$
Here $x\in X$, $s\in[0,1]$, $t\in [-R,R]$.
Because $fg=gf$, there is a fibration  $p\co E_{f,g}^{1,R}\to T^{1,R}$ and a fibrewise symplectic
closed 2-form $\omega_{E_{f,g}^{1,R}}$ coming from the one on $X$.


Fix a generic almost complex structure $\tilde J$ on $E_{f,g}^{1,R}$ 
such that $\tilde J$ is $\omega_{f,g}^{1,R}$-tame on the fibres 
and the projection $p\co E_{f,g}^{1,R}\to T^{1,R}$ is
$(\tilde J, j^{1,R})$-holomorphic.
Let $\M(j^{1,R},\tilde J)$ be the space of all $(j^{1,R},\tilde J)$-holomorphic sections
$u\co T^{1,R}\to E_{f,g}^{1,R}$:
\begin{equation}
\label{eq:hol_sect}
du+\tilde J(u)\circ du\circ j^{1,R}=0. 
\end{equation}

 For generic $\tilde J$, this moduli space is a smooth manifold that breaks into components of different dimensions. This manifold has a canonical orientation, and
in particular its 0-dimensional part
$\M^0(j^{1,R},\tilde J)$ consists of signed points. We will now describe how these signs are defined.
Let $u\in \M^0(j^{1,R},\tilde J)$. Consider the linearised equation~(\ref{eq:hol_sect}) at $u$,
\begin{equation*}
D_u\co  C^\infty(u^*T^vE_{f,g}^{1,R})\to \Omega^{0,1}(u^*T^vE^{1,R}_{f,g}).
\end{equation*}

Here $T^vE^{1,R}_{f,g}=\ker dp$
and $u^*T^vE_{f,g}^{1,R}$ is a complex bundle over the torus $T^{1,R}$.
Because $u$ has index 0, this bundle has Chern number 0 and hence is trivial;
fix its trivialisation.
Together with the holomorphic co-ordinates $(s,t)$ on $T^{1,R}$, it
induces a trivialisation of
$\Omega^{0,1}(u^*TE^{1,R}_{f,g})=\R^{2n}$.
In this trivialisation, $D_u$ is a 0-order perturbation of the Cauchy-Riemann operator:
\begin{equation}
 \label{eq:linear_floer_torus}
D_u=\bd/\bd t +J_0\bd/\bd s+A(s,t)\co C^\infty(T^{1,R},\R^{2n}) \to C^{\infty} (T^{1,R},\R^{2n})
\end{equation}
where $A(s,t)\co T^{1,R}\to Hom(\R^{2n},\R^{2n})$.
This is the same operator as considered in Subsection~\ref{subsec:index_prob_torus},
except that now $A(s,t)$ can depend on $t$ as well as on $s$.
The operator $D_u$ is always Fredholm of index 0.

Fix, once and for all, an injective  operator of the above form, for example
$$L_\id=\bd/\bd t +J_0\bd/\bd s+\id.$$
(This one is injective by Lemma~\ref{lem:ker_asymp_oper_torus}, because $\ker (\id-e^{-J_0})=0$.)
Find a smooth homotopy of operators $L_\tau$, $\tau\in[0,1]$, from $D_u$ to $L_\id$, by deforming the 0-order part $A(s,t)$ to $\id$.

\begin{definition}[{cf.~\cite[p.~51 and Appendix A]{MDSaBook}}]
\label{def:sign_sec_torus}
For $u\in \M^0(j^{1,R},\tilde J)$, define $\sgn(u)\coloneqq (-1)^\epsilon$ where
$$
\epsilon =\sum_{\tau\in[0,1]}\dim_\R\ker L_\tau.
$$
\end{definition}

For $u\in \M^0(j^{1,R},\tilde J)$, denote $\omega(u)\coloneqq \int_{T^{1,R}}u^*\omega^{1,R}_{E_{f,g}}$. The following is well known.

\begin{proposition}
\label{prop:count_sec_torus_invar}
$$
\sharp  \M^0(j^{1,R},\tilde J)\coloneqq \ \sum_{\mathclap{{u\in \M^0(j^{1,R},\tilde J)}}}\ \  \sgn(u)\cdot q^{\omega(u)}
$$
is independent of the complex structure $j^{1,R}$ on the torus and of generic $\tilde J$.\qed
\end{proposition}

\subsection{Gluing the fibration over the cylinder to the fibration over the torus} 
Given a symplectomorphism $f\co X\to X$, we have constructed
 a fibration
$p\co E_f\to S^1\times \R$ (\ref{eq:E_f}); also,
given two commuting symplectomorphisms $f,g\co X\to X$ and a parameter $R\in \R$,  we have 
constructed a fibration $E_{f,g}^{1,R}\to T^{1,R}$.
The fibres of both fibrations are symplectomorphic to $X$. Now, there is a  map  
\begin{equation}
\label{eq:glue_fib}
E_f\supset p^{-1}(S^1\times[-R,R])\to E_{f,g}^{1,R}
\end{equation}
It glues the boundary component $p^{-1}(S^1\times\{R\})$
to the other boundary component $p^{-1}(S^1\times\{-R\})$
via the symplectomorphism $g\co X\to X$ applied fibrewise along $S^1$.

Fix regular $J_s,H_s$ (\ref{eq:h_j_period_cyl}).
As in (\ref{eq:j_h_pull_g}), set 
$$J_s'=g^*J_s,\quad H_s'=H_s\circ g.$$
Choose a homotopy
$J_{s,t}$, $H_{s,t}$ (\ref{eq:cont_htopy_support}) between $J_s',H_s'$ and $J_s,H_s$.
This homotopy must be $t$-independent for  large and  small $t$;
we assume for convenience 
$$
J_{s,t}\equiv J_s',\ H_{s,t}\equiv H_s'\ \mathrm{for}\ t\le-R,
\quad \mathrm{and}\quad  J_{s,t}\equiv J_s,\ H_{s,t}\equiv H_s\ \mathrm{for}\ t\ge R.
$$
Finally, let 
$\tilde J\coloneqq \tilde J(J_{s,t},H_{s,t})$ 
be the almost complex structure
on $E_f$
from Subsection~\ref{subsec:hol_sect}, which has the property that solutions to Floer's continuation  equation are
exactly $(j_{S^1\times\R},\tilde J)$-holomorphic sections $S^1\times \R\to E_f$.


By definition, $\tilde J|_{p^{-1}(S^1\times\{R\})}$
is the $g$-pullback of $\tilde J|_{p^{-1}(S^1\times\{-R\})}$,
which agrees with the gluing (\ref{eq:glue_fib}).
So $\tilde J$
defines a glued almost complex structure $\gl \tilde J$
on $E_{f,g}^{1,R}$. 
Let us recall our notation one more time.
$\M(x,y;J_{s,t},H_{s,t})$ consists of holomorphic sections over $S^1\times \R$ which are
 solutions to Floer's continuation equation (\ref{eq:Floer_cont}), and
 $\M(j^{1,R},\gl \tilde J(J_{s,t},H_{s,t}))$ consists of holomorphic sections over the torus $T^{1,R}$.
We come to an important proposition, of which everything but formula (\ref{eq:sign_compare}) is well known.

\begin{proposition}
\label{prop_gluing}
For each $A>0$ there is $R>0$ such that there is a bijection called the gluing map and denoted by $\gl$:
$$
\gl\co \bigsqcup_{x\in \fix f_H}
\M^0(g(x),x;J_{s,t},H_{s,t})^{<A}
\ 
\xrightarrow{1-1}
\ 
 \M^0(j^{1,R},\gl \tilde J(J_{s,t},H_{s,t}))^{< A}.
$$
Here the superscripts $*\,^{< A}$ mean we are taking only those solutions whose $\omega$-area is less than $A$. The gluing map
preserves  $\omega$-areas: 
$$
\int_{S^1\times \R} u^*\omega_{E_f} = \int_{T^{1,R}}\gl(u)^*\omega_{E_{f,g}^{1,R}}
$$
and changes the signs from Definitions~\ref{def:sign_sec_cyl},~\ref{def:sign_sec_torus}
by $(-1)^{\deg x}$:
\begin{equation}
\label{eq:sign_compare} 
\sgn(u)=\sgn(\gl(u))\cdot(-1)^{\deg x}
\end{equation}
Here $u\in \M^0(g(x),x;J_{s,t},H_{s,t})^{< A}$,
and $\deg x$ is defined in Definition~\ref{def:floer_grading}.
\end{proposition}

\begin{proof}
The existence of the bijection $\gl$ is well known.
The  map $\gl$ is constructed for the case $f=g=\id$  in \cite{SchwThesis}, see also \cite{BeRa95}, and that proof carries over to arbitrary $f,g$. Alternatively, one can adopt general SFT gluing and compactness theorems \cite{CompSFT03}.

Let $u(s,t)\in \M(g(x),x;J_{s,t},H_{s,t})$. By a smooth homotopy this section can be made $t$-independent for $t$ close to $-\infty$ and $+\infty$. We can glue it into a smooth section over $T^{1,R}$ by applying (\ref{eq:glue_fib}). The smooth section over $T^{1,R}$ we obtained is smoothly homotopic to $\gl(u)$ and hence has the same $\omega$-area as $\gl(u)$: so gluing preserves $\omega$-areas.

Let us explain why $\gl$ changes the sign by $(-1)^{\deg x}$. 
We have illustrated our argument by an informal diagram below; its arrows correspond to homotopies between Fredholm operators, and  its labels are the signs determined by the mod 2 count of the dimensions of kernels appearing during the homotopies.
$$
\xymatrix{
\genfrac{}{}{0pt}{}{over}{S^1\times \R}&
 D_u \ar [rr]_-{}^-{\sgn(u)} 
& \ar@{}[d]^(.10){}="a"^(.70){}="b"  \ar@{=>}"a"; "b"_{gluing}& \bd/\bd t+J_0\bd/\bd s+A(s) & \\
\genfrac{}{}{0pt}{}{over}{T^{1,R}}&  D_{\gl (u)} \ar [rr]_-{}^-{\sgn(u)} 
\ar@<-2pt> `d[rrrr] `[rrrr]_-{\sgn(\gl(u))} [rrrr] 
&& \bd/\bd t+J_0\bd/\bd s+A(s) \ar [rr]_-{}^-{(-1)^{\deg x}} && \bd/\bd t+J_0\bd/\bd s+\id
}
$$
Take $u(s,t)\in \M^0(g(x),x;J_{s,t},H_{s,t})$ and consider  linearised Floer's operators
(\ref{eq:linear_floer}) and (\ref{eq:linear_floer_torus}):
\begin{eqnarray*}
D_u\co  H^{1,p}(S^1\times \R,\R^{2n}) \to L^p(S^1\times \R,\R^{2n}),\\
D_{\gl (u)}\co  C^\infty(T^{1,R},\R^{2n})\to C^\infty(T^{1,R},\R^{2n}).
\end{eqnarray*}
Take a  homotopy $L_\tau$ from $D_u$ to the operator (\ref{eq:L_x_cyl}) $L_{A(s)}=\bd/\bd t+J_0\bd/\bd s+A(s)$.
By Definition~\ref{def:sign_sec_cyl},
\begin{equation}
\label{eq:signs_proof_0}
\sgn(u)=(-1)^{\sum_{\tau} \dim \ker L_\tau}
\end{equation}

Let $L_\tau^\gl$ be a homotopy from $D_{gl(u)}$ to the analogous operator $L_{A(s)}=\bd/\bd t+J_0\bd/\bd s+A(s)$ 
over the torus, considered in Subsection~\ref{subsec:index_prob_torus}. It is well known that 
\begin{equation}
\label{eq:signs_proof_1}
\sum_{\tau\in[0,1]} \dim \ker L_\tau \equiv \sum_{\tau\in[0,1]} \ker L_\tau^\gl \mod 2.
\end{equation}
(This is a special case of the fact that orientations of  moduli spaces of pseudo-holomorphic sections before gluing canonically define orientations on moduli spaces after gluing.)

Take a homotopy $L_{A_\tau(s)}$ from $L_{A(s)}$ to $L_{\id}=\bd/\bd t+J_0\bd/\bd s+\id$.
To compute the kernels swept by this homotopy, we will use Lemma~\ref{lem:sign_count_torus_compute}.
First, let $\Psi(s)$ be the matrix which solves (\ref{eq:psi}) with respect to our given $A(s)$, then $\Psi(1)=df_H(X)$ by Lemma~\ref{lem:asymp_oper_cyl_solutions}. By Definition~\ref{def:floer_grading}, $\sgn \det(\id-\Psi(1))=\deg x$. Second, let $\Psi(s)$ instead be the matrix which solves (\ref{eq:psi}) with respect to the $s$-independent matrix $A(s)\equiv \id$;
the solution is $e^{-sJ_0}$, and for it we obtain $\sgn \det(\id-\Psi(1))=+1$.
Now by Lemma~\ref{lem:sign_count_torus_compute},
\begin{equation}
\label{eq:signs_proof_2}
\sum_{\tau\in[0,1]} \dim \ker L_{A_\tau(s)} \equiv \deg x \mod 2.
\end{equation}
The concatenation of homotopies $L_{A_\tau(s)}$ and $L_\tau^\gl$ is a homotopy from
$D_{\gl(u)}$ to $\bd/\bd t+J_0\bd/\bd s+\id$. So by Definition~\ref{def:sign_sec_torus},
\begin{equation}
\label{eq:signs_proof_3}
\sgn(\gl(u))=(-1)^{\sum_\tau \dim \ker  L_{A_\tau(s)}}\cdot (-1)^{\sum_\tau \dim \ker L_\tau^\gl}. 
\end{equation}
Combining (\ref{eq:signs_proof_0}), (\ref{eq:signs_proof_1}), (\ref{eq:signs_proof_2}), (\ref{eq:signs_proof_3})
we get
$
\sgn(u)=\sgn(\gl(u))\cdot (-1)^{\deg x}
$
which completes the proof.
 \end{proof}

\subsection{Proof of the elliptic relation}
\begin{proof}[Proof of Theorem~\ref{th:ell_rel}]
We only need to compile the previous statements.
It suffices to prove that for each $A>0$, the supertraces are equal up to order $q^A$:
$STr(f\hf)/q^A=STr(g\hf)/q^A$.
By Lemma~\ref{lem:trace_on_chains} and Proposition~\ref{prop_gluing}, for sufficiently large $R$ we have
$$
STr(g\hf)/q^A=\ \sum_{
\mathclap{\begin{smallmatrix}
 x\in \fix f_H,\\
  u\in \M^0(g(x),x;J_{s,t},H_{s,t})^{< A}
\end{smallmatrix}}
}\ \ 
(-1)^{\deg x}\cdot\sgn(u)\cdot q^{\omega(u)}
=\\
\sharp \M^0(j^{1,R},\tilde J(J_{s,t},H_{s,t}))^{< A}.
$$
One can repeat all constructions after swapping $f$ and $g$ to get
$$
STr(f\hf\co  HF^*(g)\to HF^*(g))=
\sharp \M^0(j^{R,1},\tilde J_1)^{< A}.
$$
Here $j^{R,1}=j^{1,\frac 1 R}$ is another complex structure on the torus (which is ``long'' in the $s$-direction, while $j^{1,R}$ is ``long'' in the $t$-direction), and $\tilde J_1$ some other almost complex structure on the total space.
Now Theorem~\ref{th:ell_rel} follows from Proposition~\ref{prop:count_sec_torus_invar}.
\end{proof}

\subsection{Finite order symplectomorphisms}
We will now prove two lemmas about the action on Floer cohomology when one of the two commuting symplectomorphisms has finite order, and derive
Proposition~\ref{prop:hf_bound_fixp_notrans}.
The proof of the next lemma is an extension of \cite[Lemma 7.1]{HoSa95}.

\begin{lemma}
\label{lem:action_hf_homology_fixloc}
Let $X$ be a symplectic manifold satisfying the $W^+$ condition. Let $g,\phi\co X\to X$ be two commuting symplectomorphisms. Suppose 
$\phi^k=\id$ and
the fixed point set $X^\phi$ is a smooth manifold (maybe disconnected, with components of different dimensions).
Then
$$
 STr(g\hf\co  HF^*(\phi)\to HF^*(\phi))=L(g|_{X^\phi})\cdot q^0+\sum_{i}a_i\cdot q^{\omega_i}, \quad \omega_i>0.
$$ 
\end{lemma}

In other words, $STr(g\hf)\in \Lambda$ contains only summands with non-negative powers of $q$, and the $q^0$-coefficient is the topological Lefschetz number of $g|_{X^\phi}$.
Using the elliptic relation we will later show that the higher order terms $a_iq^{\omega_i}$ actually vanish; this is however a separate argument and we first  prove the lemma as stated. 

\begin{proof}
First we construct a Hamiltonian function on $X$ of special form. Let $U(X^\phi)$ be a $\phi$-equivariant tubular neighbourhood of $X^\phi$,  $p\co U(X^\phi)\to X^\phi$ the projection and $\dist$ a $\phi$-invariant function on $U(X^\phi)$ measuring the distance to $X^\phi$ in some $\phi$-invariant metric. Let $H_0$ be an arbitrary function on $X^\phi$.
Define 
\begin{equation*}
 H\coloneqq H_0\circ p+\dist^2.
\end{equation*}
This is a function on $U(X^\phi)$. Extend this function to  $X$ in any way and then average it with respect to $\phi$ (this will not change the function on $U(X^\phi)$). We denote the result by $H$ again.
Note that $H|_{X^\phi}=H_0$ and $Crit(H_0)=Crit(H)\cap X^\phi$.
For the rest of the proof, $H$ will be a generic function constructed this way; in particular $H|_{X^\phi}$ is also generic.

Because $\phi$ has finite order, we can choose a $\phi$-invariant  compatible almost complex structure $J$ on $X$ which preserves $TX^\phi$, and such that $J|_{X^\phi}$ is arbitrary.
Since $J,H$ are $\phi$-invariant, they satisfy (\ref{eq:h_j_period_cyl}),
 with $f=\phi$.
Thus Floer's equation (\ref{eq:Floer_dif}) makes sense for such $s$-independent data $J,H$.
Denote $J'\equiv g^* J$, $H'\equiv H\circ g$ as in (\ref{eq:j_h_pull_g}). 

Choose an $s$-independent homotopy (\ref{eq:cont_htopy_support}) $H_t,J_t$ from $H'$ to $H$ (resp.~from $J'$ to $J$). For every $t$,  $H_t,J_t$ must be $\phi$-invariant, and as earlier 
\begin{equation}
\label{eq:h_t_quadratic}
H_t=(H_0)_t\circ p+\dist^2
\end{equation}
on $U(X^\phi)$ where $(H_0)_t=(H_t)|_{X^\phi}$ can be arbitrary.
Note that
in general, it might not be possible to find $s$-independent $J_t,H_t$  that would make all solutions
of Floer's continuation equation  (\ref{eq:Floer_cont}) regular.
However, using \cite{HoSa95} we will now argue that some solutions  of (\ref{eq:Floer_cont}) (namely gradient flowlines of $H_t$) are still generically regular.

Recall that $J_t$ defines the time-dependent metric $\omega(\cdot,J_t\,\cdot)$ on $X$ by
definition of a compatible almost complex structure. If $H$ is a function on $X$, its gradient and Hamiltonian vector fields are related by: $\nabla H=JX_H$. So $s$-independent solutions $u(s,t)\equiv x(t)$ of Floer's continuation equation (\ref{eq:Floer_cont}) are exactly 
$\omega(\cdot,J_t\, \cdot)$-gradient flowlines of $H_t$:
$$
dx(t)/dt - \nabla H_t=0.
$$
The $\phi$-periodicity condition (\ref{eq:u_period_cyl}) now reads $\phi(x(t))=x(t)$
so we are looking only at gradient flowlines inside $X^\phi$.
Note that every $s$-independent solution $u(s,t)\equiv x(t)$ of (\ref{eq:Floer_cont})  has zero area:
$\omega(u)=0$.
 Recall that solutions of (\ref{eq:Floer_cont}) are elements of $\M(x,y;J_t,H_t)$
where $x\in \fix \phi_{H'}$ and $y\in \fix \phi_H$.
Also note that $\fix \phi_H=Crit(H|_{X^\phi})$, and similarly $\fix \phi_{H'}=Crit(H'|_{X^\phi})$.

The following two facts are proved in \cite{HoSa95}  when $H_t$, $J_t$ are $t$-independent and $\phi=\id$ (that paper is interested in the equations for  Floer's differential rather than continuation maps). The proofs are valid in the general case. For example, one can track that the periodicity condition~(\ref{eq:h_j_period_cyl}), which is the only place where $\phi$ explicitly appears, is not used in the proof of the facts below.

\begin{enumerate}
\item For any $J_t,H_t$ as above, an $s$-independent solution $u(s,t)\equiv x(t)$ of (\ref{eq:Floer_cont}) is regular, i.e.~$D_u$ (\ref{eq:linear_floer}) is onto, if and only if the $\omega(\cdot,J_t\, \cdot)$-gradient flow of $H_t$ is Morse-Smale near $X^\phi$
\cite[Corollary~4.3, Theorem~7.3]{SaZe92}, compare \cite[proof of Theorem~6.1]{HoSa95}.

\item There is $\epsilon>0$ such that every solution $u(s,t)$ of (\ref{eq:Floer_cont})
with $\omega(u)<\epsilon$ is $s$-independent \cite[Lemma 7.1]{HoSa95}.
\end{enumerate}

We claim that the gradient flow of a generic $H_t$ constructed above is  Morse-Smale near $X^\phi$.
Indeed, we can choose $H_t|_{X^\phi}$ freely, so we can make the flow of $H_t|_{X^\phi}$ 
 Morse-Smale.
 Because $H_t$ is quadratic in the normal direction to $X^\phi$ (\ref{eq:h_t_quadratic}), the stable manifolds of $H_t$ are, near $X^\phi$, normal disk bundles over those of $H_t|_{X^\phi}$, and the unstable manifolds of $H_t$ lie in $X^\phi$ and coincide with those of $H_t|_{X^\phi}$. Consequently, $H_t$ is Morse-Smale near $X^\phi$ if and only if $H_t|_{X^\phi}$ is Morse-Smale.

By Remark~\ref{rem:morse_cont} or  \cite[4.2.2]{SchwBook},
\begin{equation}
\label{eq:count_zero_area_flowlines}
\sum_{
\mathclap{\begin{smallmatrix}
 x\in \fix \phi_H,\\
 u\in \M^0(g(x),x;J_{s},H_{s})\co \omega(u)\le 0
\end{smallmatrix}}
}\ \ \ 
(-1)^{\deg x}\cdot \sgn(u) \cdot q^{\omega(u)}=L(g|_{X^\phi})\cdot q^0.
\end{equation}

Although the left hand side looks exactly like the expression for $STr(g\hf)$ from Lemma~\ref{lem:trace_on_chains},
$J_{t},H_{t}$ need not be regular for all continuation equation solutions,
while $g\hf$ must be computed using a regular Hamiltonian and almost complex structure. 
To cure this, we slightly perturb $J,H$ and $J_t,H_t$ by allowing them to depend on $s$, to get $J_s,H_s$ and $J_{s,t},H_{s,t}$.
For a generic such perturbation, all solutions to (\ref{eq:Floer_cont}) with respect to $J_{s,t},H_{s,t}$ become regular.
Because $s$-independent solutions in $\M(x,y;J_t,H_t)$ were already regular, 
they are in 1-1 correspondence (via the continuation map) with some solutions in $\M(x,y;J_{s,t},H_{s,t})$ of zero $\omega$-area.
By item (2) above, every $u\in \M(x,y;J_{s,t},H_{s,t})$ with $\omega(u)<\epsilon$ actually has zero area and corresponds to an $s$-independent solution in $\M(x,y;J_t,H_t)$.
(See \cite[proof of Proposition~7.4]{HoSa95} for this argument.)
In view of (\ref{eq:count_zero_area_flowlines}) this means
$$
\sum_{
\mathclap{\begin{smallmatrix}
 x\in \fix \phi_H,\\
 u\in \M^0(g(x),x;J_{s,t},H_{s,t})\co \omega(u)\le 0
\end{smallmatrix}}
}\ \ \ 
(-1)^{\deg x}\cdot \sgn(u)\cdot q^{\omega(u)}=L(g|_{X^\phi})\cdot q^0.
$$
Lemma~\ref{lem:action_hf_homology_fixloc} follows from this equality and
Lemma~\ref{lem:trace_on_chains}.
\end{proof}

\begin{lemma}
\label{lem:str_le_dim_hf}
Let $X$ be a symplectic manifold satisfying the $W^+$ condition. Let $g,\phi\co X\to X$ be two commuting symplectomorphisms. Suppose 
$\phi^k=\id.$
Then
$$
 STr(\phi\hf\co  HF^*(g)\to HF^*(g))=a\cdot q^0,\quad \text{where }a\in \C \text{ and } |a|\le \dim_\Lambda HF^*(\phi).
$$ 
\end{lemma}

\begin{proof}
By Lemma~\ref{lem:action_powers} $(\phi\hf)^k=\id$, so all eigenvalues of $\phi\hf$ 
are among the roots of unity $\sqrt[k]{1}\cdot q^0\in\Lambda$. 
The signed sum of these eigenvalues gives $STr(\phi\hf)$, and Lemma~\ref{lem:str_le_dim_hf} follows.
\end{proof}

The elliptic relation (Theorem~\ref{th:ell_rel}) and Lemma~\ref{lem:str_le_dim_hf}
imply the following corollary.

\begin{corollary}
\label{cor:str_high_order_vanish}
The terms $a_i\cdot q^{\omega_i}$, $\omega_i>0$ from Lemma~\ref{lem:action_hf_homology_fixloc} actually vanish.\qed
\end{corollary}

\begin{proof}[Proof of Proposition~\ref{prop:hf_bound_fixp_notrans}]
The proposition follows from Lemma~\ref{lem:action_hf_homology_fixloc},
Lemma~\ref{lem:str_le_dim_hf} and Theorem~\ref{th:ell_rel}.
\end{proof}

\begin{remark}
 \label{rem:alt_proof_bound_main}
As promised in Remark~\ref{rem:alt_proof_bound_intro} we sketch an alternative proof of Proposition~\ref{prop:hf_bound_fixp_notrans} which does not appeal to Theorem~\ref{th:ell_rel}.
Suppose for simplicity a symplectomorphism $f\co X\to X$ commutes with a symplectic involution $\iota$ and
$f$ has non-degenerate isolated fixed points. Note that, for general reasons, $d\iota$ acts by $-\id$ on the normal bundle to its fixed locus $X^\ii$. For computing $HF^*(f)$, choose the zero Hamiltonian perturbation and an almost complex structure which is $\ii$-invariant at points $ x\in \fix f\cap X^\ii$.
Then $\ii\hf$ only counts constant solutions $u(s,t)\equiv x\in \fix f\cap X^\ii$. (Because $f$ has isolated fixed points, the only zero-area solutions are constant, and because $\iota\hf^2=\id$, all positive area solutions cancel.)
However, the sign associated to a constant solution $u$ is not always positive. The reason is that we must write the linearised Floer's operator $D_u$ in 
a trivialisation of $u^*T_xX=S^1\times \R\times T_xX$ which differs by $d\iota(x)$ over the two ends of the cylinder,
according to the definition in Subsection~\ref{subsec:coherent_orient_cont}.
Consider the splitting $T_xX=T_xX^\ii\oplus N_xX^\ii$ into the $+1$ and $-1$ eigenspaces of $d\ii(x)$.
We can choose the constant trivialisation of $u^*T_xX^\ii$ 
and get the $\R$-independent operator on this subspace, which by definition carries the positive sign.
However, we are not allowed to choose the constant trivialisation of $u^*N_xX^\ii$ (instead, an allowed choice is, for example, a rotation from $\id$ to $-\id$ with parameter $t$), so $D_u$ will not be the canonical $\R$-invariant operator on $N_xX$ and can carry a nontrivial sign from Definition~\ref{def:sign_sec_cyl}.
We claim that this sign equals $\sgn\det (\id-df(x)|_{N_xX^\ii})$.
The computation can be essentially be reduced to the index problem considered in Subsection~\ref{subsec:index_prob_torus}, since $D_u$ can still be chosen independent of one variable; a related Lagrangian version
of this statement is \cite[Lemma~14.11]{SeiBook08}.
Once the signs are known, it is easy to see that $STr(\ii\hf)=L(f|_{\fix \ii})\cdot q^0$:
\begin{align*}
STr(\ii\hf)&=\sum_{x\in \fix f\cap X^\ii}(-1)^{\deg x}\cdot \sgn \det(\id-df(x)|_{N_xX^\ii})\cdot q^0
\\
&=\sum_{x\in \fix f\cap X^\ii} \sgn \det(\id-df(x)|_{T_xX^\ii})\cdot q^0\ \ = \ \ L(f|_{\fix \ii})\cdot q^0.
\end{align*}
The bound $\dim HF^*(f)\ge L(f|_{\fix \ii})$ follows as in Lemma~\ref{lem:str_le_dim_hf}. 
\end{remark}

\subsection{Lagrangian elliptic relation}
\label{subsec:ell_rel_lag} 
In this subsection, we briefly explain Theorem~\ref{th:ell_rel_Lag} and Proposition~\ref{prop:hf_bound_lag}.
Let $X$ be a monotone  symplectic manifold, i.e.~$[\omega(X)]=\lambda c_1(X)$
as elements of $H^2(X;\R)$, $\lambda > 0$.
Let $\phi\co X\to X$ be a symplectomorphism, and $L_i\subset X$ be two connected monotone Lagrangian submanifolds  such that $\phi(L_i)=L_i$.

In order to define the action $\phi\hf\co HF^*(L_1,L_2)\to HF^*(L_1,L_2)$ over a field of characteristic not equal to two, we must fix the following additional data.
First, $L_i$ must be oriented, although $\phi$ need not preserve
the orientations.
(In Appendix~\ref{app:growth_ring} we use the orientation-reversing case.)
Second, the hypothesis below must be satisfied.

\begin{hypothesis}
\label{hyp:equiv_spin}
$L_i$ must be equipped with spin structures $S_i$
together with isomorphisms $\phi^*S_i\to S_i$ if $\phi|_{L_i}$ preserves orientation,
and $\phi^*S_i\to \bar S_i$ if $\phi|_{L_i}$ reverses orientation.
Here $\bar S_i$ is the following spin structure on $\bar L_i$ (that is, on $L_i$ with the opposite orientation). 
The original spin structure $S_i$ is a trivialisation of $TL_i$ over the 1-skeleton of $L_i$ which extends over the 2-skeleton and agrees with the orientation on $L_i$. By definition, $\bar S_i$ is the composition of the trivialisation $S_i$
with a fixed orientation-reversing isomorphism $\R^n\to \R^n$, for example the one which multiplies the first co-ordinate by 
 $-1$. We note the desired isomorphisms $\phi^*S_i\to S_i$ or $\phi^*S_i\to \bar S_i$ always exist if $L_i$ are simply-connected.
\end{hypothesis}

In \cite[Section~14]{SeiBook08}, similar data (defined only for an involution $\phi$, with an extra condition on the ``squares'' of the above isomorphisms, but also allowing non-orientable Lagrangians) was called an equivariant Pin structure.

Pick some $J_s,H_s$ defining  Floer cohomology $HF^*(L_1,L_2;J_s,H_s,S_i)$, see \cite{Oh93,FO3Book} for a definition in the monotone setting.
We have included the choice of spin structures in our notation.
The action $\phi\hf$ is the composition
$HF^*(L_1,L_2;J_s,H_s,S_i)\to HF^*(L_1,L_2;\phi^*J_s,H_s\circ \phi,\phi^*S_i)\to HF^*(L_1,L_2;J_s,H_s,S_i)$.
Here the first map is the tautological chain-level map that takes all chain generators and Floer's solutions to their $\phi$-image; we are using that $\phi L_i=L_i$. 
The second one is the continuation map. We skip the proof of the next lemma.

\begin{lemma}[{cf.~\cite[Sections (14a)~and (14e)]{SeiBook08}}]
\label{lem:lag_ell_rel_square}
If $\phi^k=\id$ then $(\phi\hf)^k=\pm\id$.\qed
\end{lemma}

Note that, unlike Lemma~\ref{lem:action_powers} and \cite[top of p.~310]{SeiBook08}, we do not necessarily get $(\phi\hf)^k=\id$, but having $(\phi\hf)^k=\pm\id$ is enough for future applications.

Choose $J_s,H_s$ (\ref{eq:h_j_period_cyl}) to define Floer's complex $CF^*(\phi;J_s,H_s)$. 
Take the fibration  $p\co E_\phi\to S^1\times[0,+\infty)$
with monodromy $\phi$ around the circle as in (\ref{eq:E_f}), but now over the semi-infinite cylinder
$S^1\times [0,+\infty)$ instead of $S^1\times\R$.
It contains the ``boundary condition'' manifold $S^1\times L\subset p^{-1}(S^1\times \{0\})$.
The symplectic form on $X$ defines a fibrewise symplectic form $\omega_{E_\phi}$ on $E_\phi$.
Choose a tame almost complex structure $\tilde J$ on $E_\phi$ which, over
$S^1\times [1,+\infty)$, equals $\tilde J(J_s,H_s)$ for some $J_s,H_s$ (see Subsection~\ref{subsec:hol_sect}),
and in particular is independent of $t\in[1,+\infty)$. 

Take  $x\in \fix\phi_H$ (\ref{eq:f_H}), that is, a generator of $CF^*(\phi;J_s,H_s)$.
We define $\M^0(L,x)$ to be the set of all zero index
$\tilde J$-holomorphic sections $u(s,t)\co S^1\times [0,+\infty)\to E_\phi$ which are
asymptotic, as $t\to +\infty$, to the Hamiltonian trajectory $\psi_s(x)$ (\ref{eq:hs_flow}),
and satisfy the Lagrangian boundary condition $u(s,0)\in S^1\times L$.
Then we define 
$$
[L]^\phi=\sum_{x\in \fix \phi_H}\sum_{u\in\M^0(L,x)}  \pm q^{\omega(u)}\cdot [x]\in HF^*(\phi).
$$
Here $[x]\in HF^*(\phi)$ is the cohomology class of the chain generator $x$, and $\omega(u)=\int_{S^1\times [0,+\infty)} u^*\omega_{E_\phi}$.
The signs are defined using the chosen spin structures on $L_i$ and coherent orientations for $\phi$.
This is a version of the open-closed string map, cf.~\cite{RiSmi12}.

Next we review the quantum  product $HF^*(\phi)\otimes HF^*(\phi^{-1})\to HF^*(\id)\cong QH^*(X)$.
It counts holomorphic sections of a symplectic fibration over $S^2$ with three punctures and monodromies $\phi,\phi^{-1},\id$ around them.
The first two punctures serve as inputs for $HF^*(\phi)$, $HF^*(\phi^{-1})$, and the third puncture is the output, see \cite{MDSaBook} for details.
If one caps the output puncture by a disk,
the count of sections over the resulting twice-punctured sphere (see the lower part of Figure~\ref{fig:lag_ell_rel}(a)), gives the
composition $HF^*(\phi)\otimes HF^*(\phi^{-1})\to HF^*(\id)\stackrel{\chi}\to \Lambda$
of the product and the integration map $\chi$ (once we identify $HF^*(\id)$ with $QH^*(X)$).

\begin{figure}[h]
\includegraphics{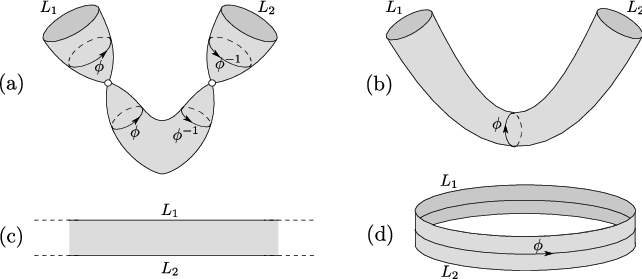}
\caption{Proving the Lagrangian elliptic relation.}
\label{fig:lag_ell_rel}
\end{figure}

Combining the definitions,
$\chi([L_1]^\phi *[L_2]^{\phi^{-1}})$
counts holomorphic sections over two cylinders and a twice-punctured sphere which
have the same asymptotics over the punctures in two pairs, see Figure~\ref{fig:lag_ell_rel}(a).
Here the cylinder $S^1\times [0,+\infty)$ is seen as a once punctured disk.
This count equals the number of sections of a glued fibration over an annulus, with monodromy $\phi$ around the core circle, and Lagrangian conditions  $S^1\times L_1, S^1\times L_2$ over the boundary of the annulus. The annulus carries a fixed ``long'' complex structure, see Figure~\ref{fig:lag_ell_rel}(b).

On the other hand, $STr(\phi\hf)$ counts sections of a trivial fibration over the strip $[0,1]\times \R$
with Lagrangian boundary conditions $\R\times L_i$ and asymptotics differing by $\phi$
over $t\to \pm\infty$, see Figure~\ref{fig:lag_ell_rel}(c).
We can glue the fibration over the strip twisting it by $\phi$ to get a fibration over the annulus 
which we have already encountered: it carries Lagrangian conditions $S^1\times L_i$ over the boundary and has monodromy $\phi$ around the core circle, see Figure~\ref{fig:lag_ell_rel}(d).
By gluing, $STr(\phi\hf)$ is equal to the count of holomorphic sections of this fibration, with a fixed (``long'', but in the other direction than before) complex structure on the annulus.
As the count of sections does not depend on the complex structure on the annulus, we get Theorem~\ref{th:ell_rel_Lag}.
We omit the discussion of signs which was carried out in detail for the case of commuting symplectomorphisms. The signs in present case can be studied by similar arguments if we superficially deform the Lagrangians so that $T_pL_1=T_pL_2$ for all intersection points $p\in L_1\cap L_2$, keeping these points isolated,
and then pick non-degenerate Hamiltonians $H_1,H_2$ to compute $HF^*(L_1,L_2)$.

Let us now explain Proposition~\ref{prop:hf_bound_lag}.
The most important step is to prove a Lagrangian analogue of Lemma~\ref{lem:action_hf_homology_fixloc}:
if $\phi$ is a map of finite order with fixed locus $X^\phi$ and smooth orientable Lagrangian fixed loci $L_i^\phi\subset X^\phi$ then
\begin{equation}
\label{eq:inv_act_on_inv_lag}
\chi([L_1]^\phi *[L_2]^\phi)=([L_1^\phi]\cdot[L_2^\phi])\cdot q^0+\sum_i a_i\cdot q^{\omega_i},\quad \omega_i>0. 
\end{equation}

Recall that $[L_1^\phi]\cdot[L_2^\phi]\in\Z$ is the homological intersection of the fixed loci $L_1^\phi,L_2^\phi$
inside $X^\phi$. (Note that $L_i^\phi$ are automatically isotropic but not necessarily Lagrangian, although we will only use the case when they are Lagrangian. One can get examples of $(\phi\hf)^k=-\id$ in Lemma~\ref{lem:lag_ell_rel_square} when dimensions of $L_1^\phi$, $L_2^\phi$ are different.)

In order to count sections of the configuration on 
Figure~\ref{fig:lag_ell_rel}(a),
we must specify the data $J_{s,t},H_{s,t}$
over our configuration consisting
of two half-cylinders $S^1\times[0,+\infty)$
and a twice-punctured sphere which we will now see as the cylinder $S^1\times \R$.
Similarly to Lemma~\ref{lem:action_hf_homology_fixloc}, we choose the data to be of special form,
namely independent of the basepoint: $J_{s,t}\equiv J$, $H_{s,t}\equiv H$ (this forces $J,H$ to be $\phi$-equivariant).
With this data, $s$-independent ($s\in S^1$) sections become gradient flowlines of the Morse function  $H$ inside the fixed locus $X^\phi$.
Rigid sections over $S^1\times \R$ are constant, while rigid sections over $S^1\times [0,+\infty)$
are flowlines from $L_i$ to a critical point of $H$.
This way, the count of $s$-independent rigid configurations
on 
Figure~\ref{fig:lag_ell_rel}(a)
is $\sum_{x\in \textit{Crit}^{\,n}(H|_{X^\phi})} ([L_1^\phi]\cdot [\textit{Stab}(x)])\, ([L_2^\phi]\cdot [\textit{Stab}(x)])$
where $\textit{Crit}^{\,n}$ are index $n$ critical points, $n=\half\dim_\R X^\phi$, and $\textit{Stab}$ are stable manifolds in $X^\phi$.
This sum equals the intersection $[L_1^\phi]\cdot[L_2^\phi]$.

Finally, one must argue that these configurations of flowlines are regular,
and are the only zero area solutions.
(There could be other positive area solutions which are not necessarily regular).
This is a variation on the lemmas cited in the proof of Lemma~\ref{lem:action_hf_homology_fixloc}.
Then one makes the data $J,H$ regular by allowing them to depend on $s,t$
and argues that the count of zero area solutions (which were already regular) is preserved.

On the other hand, if $\phi$ is of finite order then $\phi\hf\co  HF^*(L_1,L_2)\to HF^*(L_1,L_2)$ is of finite order by Lemma~\ref{lem:lag_ell_rel_square}, and the eigenvalues of $\phi\hf$ are among $\sqrt[2k]{1}\cdot q^0$. Consequently, $STr(\phi\hf)=a\cdot q^0$, $|a|\le \dim_\Lambda HF^*(L_1,L_2)$.
Now Theorem~\ref{th:ell_rel_Lag} and formula (\ref{eq:inv_act_on_inv_lag}) imply Proposition~\ref{prop:hf_bound_lag}.

\end{section}



\begin{section}{Vanishing spheres and Dehn twists}
\label{sec:van_spheres_intro}
Let $Y$ be a K\a hler manifold
with a K\a hler  form $\omega$, and $\cL\to Y$ a very ample holomorphic line bundle.
Let $X\subset Y$ be a smooth divisor in the linear system $|\cL|$. 
In this section we define 
$|\cL|$-vanishing Lagrangian spheres in the symplectic manifold $(X,\omega|_X)$.
They exist if the line bundle $\cL\to Y$ has zero defect (see below) and are then unique up to symplectomorphism.
Throughout this section, we denote by $D\subset \C$ the unit complex disk.

\subsection{Lefschetz fibrations and vanishing cycles}
\label{subsec:lef_fib_van_Cyc}
This subsection reviews well known material, see e.g.~\cite{SeiBook08}.

\begin{definition}[Lefschetz fibration with a unique singularity]
\label{def:lef_fib}
Suppose $E$ is a smooth manifold, $\Omega$ a closed 2-form on $E$, and $\pi\co  E\to D$ is a smooth proper map. The triple $(E,\Omega,\pi)$ is called
a Lefschetz fibration with a unique singularity if there is a point $p\in E$ (without loss of generality, we assume $\pi(p)=0\in D$), and a neighbourhood $U(p)$ such that:
\begin{itemize}
 \item $\pi$ is regular outside $U(p)$, and the restriction of $\Omega$ on the regular fibres of $\pi$ is symplectic;
\item there exists a complex structure on $U(p)$  with a holomorphic chart $x_1,\ldots,x_n$ such that
$$\pi(x_1,\ldots,x_n)=x_1^2+\ldots+x_n^2;$$
\item $\Omega|_{U(p)}$ is K\"ahler with respect to the above complex structure.
\end{itemize}
\end{definition}

All smooth fibres $E_t\coloneqq \pi^{-1}(t)$ contain a Lagrangian sphere, uniquely defined up to Hamiltonian isotopy. Let us sketch its construction, as we will refer to it later in the proof of Lemma~\ref{lem:A2chain_from_fib}.
Because the smooth fibres $E_t$ are symplectomorphic to each other by parallel transport with respect to the $\Omega$-induced connection on $E$, it suffices to construct a Lagrangian sphere in $E_t$ for a small $t\in \R_+$.
Define $L\subset U(p)\subset E$ by the equation
$$
x_1^2+\ldots+x_n^2=t,\quad x_i\in \R.
$$
Clearly $L\subset E_t$ and is a Lagrangian sphere for $t\in \R_+$
with respect to the standard symplectic structure $\Omega_{std}$ on $U(p)\subset \C^n$.
However, it is generally not possible
to make our form $\Omega|_{U(p)}$ standard by a holomorphic change of co-ordinates preserving $\pi$.
Instead, we can follow the argument of \cite[Lemma~1.6]{Sei03_LES}: there is a function $f$ on $U(p)$ such that
$\Omega|_{U(p)}=\Omega_{std}+dd^cf$. We can deform $f$ to $0$ in a smaller neighbourhood $U'(p)\subset U(p)$
while leaving $f$ unchanged outside of $U(p)$.
Let $f_r$ be such a homotopy and define $\Omega_r\coloneqq \Omega$ outside of $U(p)$, and
$\Omega_r|_{U(p)}\coloneqq \Omega_{std}+dd^cf_r$. 
Observe that $\Omega_0=\Omega$ and $\Omega_1|_{U'(p)}=\Omega_{std}$.
For all $r$, the smooth fibres $(E_t,\Omega_r|_{E_t})$
are symplectic and the cohomology class of $\Omega_r|_{E_t}$ is constant, so by Moser's lemma the smooth fibres are actually symplectomorphic
to each other for any $r$.
In particular, the Lagrangian sphere $L\subset (E_t,\Omega_{\mathrm{std}}|_{E_t})$ constructed above
can be mapped by this symplectomorphism to a Lagrangian sphere in
 $(E_t,\Omega|_{E_t})$.
 
\begin{definition}[Vanishing Lagrangian sphere]
\label{def:van_sph_lef_fib}
A Lagrangian sphere in a smooth fibre $E_t$ is called vanishing for the Lefschetz fibration $E\to D$
if it is Hamiltonian isotopic to the one constructed above. 
\end{definition}

\subsection{Defect of a line bundle}

\begin{definition}[Defect of a line bundle]
\label{def:defect_line_bun}
Let $Y$ be a complex manifold and  $\cL\to Y$ a very ample holomorphic line bundle, giving 
an embedding $Y\subset (\P^N)^*$ where $\P^N=\P H^0(Y,\cL)$.
The discriminant variety $\Delta\subset \P^N$ is the dual variety to $Y$,
parameterising all hyperplanes in $(\P^N)^*$ which are tangent to $Y\subset \P^N$.
Equivalently, it parameterises all singular divisors in the linear system $\P H^0(Y,\cL)$.
The defect of $\cL$ is the number $$\mathrm{def}\, \cL=N-1-\dim \Delta\ge 0.$$
\end{definition}

Line bundles usually have zero defect; for us, it is useful to note the following.

\begin{lemma}[{\cite[page 532]{BeFaSo92}}]
\label{lem:pos_def_curve}
Suppose $\cL\to Y$ is a very ample line bundle.
If $\mathrm{def}\, \cL\ge 1$, there exists a smooth rational curve $l\subset Y$
such that $\cL\cdot l=1$.\qed
\end{lemma}

For completeness, let us sketch a proof. 
Recall that points in $\Delta^{reg}$ correspond to generic
hyperplanes $H\subset (\P^N)^*$ which are not transverse to $Y$.
If $\mathrm{def}\, \cL\ge 1$,
for such a hyperplane $H\in \Delta^{reg}$
the contact locus $(H\cap Y)^{sing}$ is a linear $\P^{\mathrm{def}\, \cL}$ \cite[Theorem~1.18]{Tev05}.
Take any line $l\cong \P^1$ in $H$.
Obviously it intersects a generic smooth hyperplane section $\tilde H\cap Y$
transversely at a single point, which means $\cL\cdot l=1$.

\begin{corollary}
\label{cor:power_zero_def}
Suppose $\cL\to Y$ is a very ample line bundle.
For any $d\ge 2$, $\mathrm{def}\, \cL^{\otimes d}=0$.\qed
\end{corollary}

\subsection{$|\cL|$-vanishing spheres in divisors}

Recall $D\subset\C$ denotes the unit disk.

\begin{definition}[Total space of a family of divisors]
\label{def:tot_fam_div}
Let $Y$ be a K\a hler manifold and $\cL\to Y$ a very ample line bundle.
Take a holomorphic embedding $u\co D\to \P H^0(Y,\cL)=|\cL|$, then
each point $t\in D$ defines a divisor $X_{u(t)}\subset Y$.
We call $\{X_{u(t)}\}_{t\in D}$ a family of divisors.
The total space of the family $\{X_{u(t)}\}_{t\in D}$ is 
$$E\coloneqq \{(x,u(t)):\  x\in X_t,\ t\in D\}\subset Y\times \P H^0(Y,\cL).$$
The restriction of the product K\a hler form from $Y\times \P H^0(Y,\cL)$ to $E$ makes $E$ a K\a hler manifold. 
There is a canonical projection $\pi\co E\to D$ whose fibres are $X_{u(t)}$.
In future we shall write $\{X_t\}_{t\in D}$ instead of $\{X_{u(t)}\}_{t\in D}$.
\end{definition}

\begin{definition}[$|\cL|$-vanishing Lagrangian sphere in a divisor]
\label{def:lin_vanish}
Let $Y$ be a K\a hler manifold and $\cL\to Y$ a very ample line bundle with zero defect, and 
with $\dim \P H^0(Y,\cL)\ge 2$.
Let $\Delta\subset \P H^0(Y,\cL)$ be the discriminant variety from Definition~\ref{def:defect_line_bun}.
Let $u\co D\to \P H^0(Y,\cL)$ be a holomorphic embedding
 such that $u(0)\in \Delta^{reg}$, $u(t)\notin \Delta$ for $t\neq 0$, and the intersection of $u(D)$ with $\Delta^{reg}$ is transverse. Let $\pi\co  E\to D$ be as in Definition~\ref{def:tot_fam_div}.

By \cite[1.8]{Lj92}, $\pi\co E\to D$ is a Lefschetz fibration with a unique singular point
over $t=0$ (in particular, $X_0$ has a single node). 
The vanishing sphere $L\subset X_1$ of this fibration is called an $|\cL|$-vanishing sphere.
\end{definition}

Obviously, every smooth divisor in the linear system $|\cL|$ contains  an $|\cL|$-vanishing sphere,
if $\cL$ has zero defect.
Two different maps $u,u'\co D \to H^0(Y,\cL)$ with  $u(1)=u'(1)$ can give two $|\cL|$-vanishing spheres in $X_1$
which are not Hamiltonian isotopic and even not homologous to each other, such as in the case of Lemma~\ref{lem:a2_chain_div}.
However,  $|\cL|$-vanishing spheres are unique up to symplectomorphism.

\begin{lemma}
 \label{lem:van_unique}
Let $\cL\to Y$ be a very ample line bundle over a K\a hler manifold $Y$, $\mathrm{def}\, \cL=0$. Suppose $X,X'$ are two smooth divisors in the linear system $|\cL|$ and $L\subset X$, $L'\subset X'$ are two $|\cL|$-vanishing Lagrangian spheres. Then there is a symplectomorphism $\psi\co X\to X'$ such that $\psi(L)=L'$.
\end{lemma}

This lemma is probably well known, but we don't have a clear reference for it,
so we prove it here. An auxiliary lemma is required.

\begin{lemma}
 \label{lem:par_trans_A1}
Let $\pi\co X\to D \times [0, 1]$ be a smooth proper map and $\Omega$ a closed 2-form on $X$. Suppose  that
for every $s \in [0,1]$, $X_{D;s} \coloneqq  \pi^{-1}(D\times \{s\})$, equipped with the restriction of $\Omega$, is a
Lefschetz fibration over $D$ with a unique singularity over $0\in D$. (In particular, the fibres of $\pi$ are symplectic.)
For $t\in D$, $s \in [0, 1]$ denote by $X_{t;s}$ the fibre $\pi^{-1}(\{t\} \times \{s\})$.
Let $L_0 \subset X_{1;0}$ (resp.
$L_1 \subset X_{1;1}$) be a vanishing sphere of the Lefschetz fibration $X_{D;0}$ (resp.~$X_{D;1}$).
Then there is a symplectomorphism $\psi\co X_{1;0} \to X_{1;1}$ such that $\psi(L_0)=L_1$.
\end{lemma}

\begin{proof}
One can choose a smooth family of Lagrangian spheres $L_s\subset X_{1;s}$ such that $L_s$ is vanishing for 
the fibration on $X_{D,s}$, and $L_0,L_1$ are the given spheres.
This is easily seen from our definition or from \cite[proof of Lemma 16.2]{SeiBook08}.

Fix $s\in [0;1]$ and let
$\phi_\epsilon \co  X_{1;s} \to X_{1;s+\epsilon}$ be the parallel transport with respect to $\Omega$ \cite[Section 15a]{SeiBook08} along the $s$-direction. 
Look at $\phi_\epsilon(L_s)$ and $L_{s+\epsilon}$: these are two Lagrangian spheres in $X_{1;s+\epsilon}$ which coincide when  $\epsilon=0$, so they remain sufficiently close to each other for $\epsilon$ small enough, say $|\epsilon|<\epsilon(s)$. Being sufficiently close,
the two spheres are Hamiltonian isotopic inside $X_{1;s+\epsilon}$. By composing $\phi_{\epsilon}$ with this Hamiltonian
isotopy, we get a symplectomorphism $\psi_{\epsilon} \co  X_{1;s} \to X_{1;s+\epsilon}$ taking $L_s$ to $L_{s+\epsilon}$. 

The open cover of $[0,1]$ consisting of the intervals $\{(s-\epsilon(s),s+\epsilon(s))\}_{s\in [0;1]}$ admits a finite subcover.
We know that for $s,s'$ within a single interval, $L_s$ can be taken to $L_s'$ by a symplectomorphism $X_{1;s}\to X_{1;s'}$; using the finite subcover, we are able to find a finite composition of such maps which is a symplectomorphism
$X_{1;0}\to X_{1;1}$ taking $L_0$ to $L_1$.
\end{proof}

\begin{proof}[Proof of Lemma~\ref{lem:van_unique}]
Let $u,u'\co D\to \P H^0(Y,\cL)$ be two holomorphic maps as in Definition~\ref{def:lin_vanish},
and denote $X=X_{u(1)}$, $X'=X_{u'(1)}$.
By Definition~\ref{def:lin_vanish}, $u(0),u'(0)\in \Delta^{reg}$. 
 Since $\Delta^{reg}$ is connected, one can find a path $\alpha(s) \in \Delta^{reg}$ from $u(0)$ to
$u'(0)$, $s\in [0,1]$. Next one can find an $s$-parametric family of holomorphic disks $u_s\co D\to \P H^0(Y,\cL)$
such that $u_0=u$, $u_1=u'$, $u_s(0)\in\Delta^{reg}$ and $u_s(D)$ intersects $\Delta^{reg}$ transversely. Consider the space 
$$E\coloneqq \{(x,u_s(t)):\   t\in D,\ s\in [0,1],\ x\in X_{u(t)}\}\subset Y\times \P H^0(Y,\cL).$$
It carries a closed $2$-form which is the restriction of the product  K\a hler form to $Y$ and $\P H^0(Y,\cL)$. 
There is also a canonical projection  $E\to D\times [0,1]$. With these data, $E$ satisfies conditions of Lemma~\ref{lem:par_trans_A1}.
This lemma provides the desired symplectomorphism $\psi\co X\to X'$ taking an given $|\cL|$-vanishing sphere in $X$
to a given one in $X'$.
\end{proof}

\subsection{Dehn twists}
\label{subsec:dehn_twists}
We recall the definition of Dehn twists from \cite[Section (16c)]{SeiBook08}.
First, one defines the Dehn twist as a compactly supported symplectomorphism of $T^*S^n$.
Fix the standard round metric on $S^n$, and let $|\xi|$ be the norm function on $T^*S^n$. It is non-smooth at the 0-section; away from the 0-section, its Hamiltonian flow is the normalised geodesic flow. Take a function $b(r)\co \R\to \R$ with compact support and such that $b(r)-b(-r)=-r$. 
The Dehn twist $\tau\co  T^*S^n\to T^* S^n$ is the $2\pi$-flow of the Hamiltonian function
$b(|\xi|)$. It extends smoothly to the 0-section by the antipodal map, thanks to the special form of $b(r)$.
As a result, $\tau$ is a compactly supported symplectomorphism of $T^*S^n$.
Its behaviour in $T^*S^n$ is well understood.
\begin{theorem}
	\label{th:twist_tstar}
	\begin{enumerate}
		\item $\tau$ has infinite order in $\Symp^c(T^*S^n)/\Ham^c(T^*S^n)$, the group of compactly supported symplectomorphisms of $T^*S^n$
		modulo compactly-supported symplectic isotopy.
		\item If $n$ is even, $\tau$ has finite order in $\pi_0 \Diff^c(T^*S^n)$, the group of compactly-supported diffeomorphisms of $T^*S^n$
		modulo compactly-supported isotopy \cite{Kry05}. \qed
	\end{enumerate}
\end{theorem}

When $n=2$ it is further known that  $\tau$ generates $\pi_0 \Symp^c(T^*S^2)\cong \Z$, and $\tau^2$ is smoothly isotopic to $\id$
in $\Diff^c(T^*S^2)$ \cite{SeiL4D},  see also \cite[Theorem~1.21]{Av12}.

Next, if $L\subset X$ is a Lagrangian sphere in any symplectic manifold,
a neighbourhood of $L$ in $X$ is symplectomorphic to a neighbourhood of the 0-section in $T^*S^n$.
So one can pull back $\tau$ using this symplectomorphism and then extend it by the identity to get a map $\tau_L\co X\to X$.
It is a symplectomorphism uniquely defined up to Hamiltonian isotopy (once a parameterisation of $L$ is fixed), supported in a neighbourhood of $L$.

\begin{definition}[Dehn twist]
	\label{def:dehn_twist}
	The symplectomorhism $\tau_L\co X\to X$ is called the Dehn twist around $L$.
\end{definition}

\begin{lemma}[Picard-Lefschetz formula, \cite{Lj92}]
	\label{lem:picard_lef}
	If $\dim_\R X=2n$ and $L\subset X$ is a Lagrangian sphere, then $(\tau_L)_*$ acts by $\id$ on $H_i(X)$, $i\neq n$.
	For any $[A]\in H_n(X)$,
	$$ (\tau_L)_*[A]=[A]-\epsilon\cdot ([L]\cdot [A])[L].$$
	Here $\epsilon= (-1)^{\frac 1 2 n(n-1)}$.
	Consequently:
	\begin{enumerate}
		\item if $n$ is even, then $(\tau_L)_*^2$ acts by $\id$ on $H_*(X)$.
		\item 
		if $n$ is odd and $[L]\in H_n(X;\R)$ is non-zero, then
		$(\tau_L)_*$ is an automorphism of infinite order of $H_*(X)$.\qed
	\end{enumerate}
\end{lemma}

Summarising
Theorem~\ref{th:twist_tstar}(2) and Lemma~\ref{lem:picard_lef}(2), we arrive to
the following well known statement.

\begin{corollary}
	\label{cor:twist_homol_order}
	Let $\dim X_\R=2n$ be a compact symplectic manifold and $L\subset X$  a Lagrangian sphere non-zero in $H_n(X;\R)$. 
	\begin{enumerate}
		\item If $n$ is even, $\tau_L$ has finite order in $\pi_0 \Diff (X)$,
		\item if $n$ is odd, $\tau_L$ has infinite order in $\pi_0\Diff(X)$. \qed
	\end{enumerate}
\end{corollary}

The next lemma relates Dehn twists and Lefschetz fibrations, see \cite[(15b)]{SeiBook08}
for details.

\begin{lemma}[\cite{Sei03_LES,SeiBook08}]
	\label{lem:dehn_tw_and_monodromy}
	Let $(E,\Omega,\pi)$ be a Lefschetz fibration with a unique singularity.
	Let $E_1$ be its regular fibre and $L\subset E_1$ a vanishing Lagrangian sphere.
	Then the Dehn twist $\tau_L\co E_1\to E_1$ is Hamiltonian isotopic to the symplectic monodromy map $E_1\to E_1$ obtained by applying  symplectic parallel transport   to the fibres $E_t$ along the circle $t\in\bd D$.\qed
\end{lemma}

\begin{remark}
	\label{rem:ord_twist_known}
	Let $X$ be a symplectic manifold and $L\subset X$  a Lagrangian sphere; assume $L$ is non-zero in $H_n(X)$.
	There are three main previously known cases when $\tau_L$ has infinite order in $\Symp(X)/\Ham(X)$ (if $X$ is non-compact, consider $\Symp^c(X)/\Ham^c(X)$ instead):
	
	\begin{enumerate}
		\item $\half \dim_\R X$ is odd, as explained above;
		\item $X$ is exact with contact type boundary, and $L$ is exact (Seidel, unpublished);
		\item  $X$ is Calabi-Yau, and there is another Lagrangian sphere $L'$ intersecting $L$ once transversely \cite{Sei00}.
	\end{enumerate}
	
	Let $X=Bl_k \P^2$ be the blowup of $\P^2$ in $k$ generic points, $2\le k\le 8$, with the monotone symplectic form, and $L\subset X$ be any Lagrangian sphere.
	Seidel \cite{SeiL4D} showed that $\tau_L$ has order~$2$ in $\Symp(X)/\Ham(X)$ when $k=2,3,4$
	and has order greater than $2$ when $k=5,6,7,8$, but did not prove it was infinite.
	Note that $X=Bl_6\P^2$  is the cubic surface $X\subset \P^3$, to which Theorem~\ref{th:twist_inf_grass} applies.
\end{remark}

\end{section}

\begin{section}{Constructing invariant Lagrangian spheres}
	\label{sec:van_sph_construct}
	The aim of this section is to state and prove
	Proposition~\ref{prop:a2_in_sym_div}, which
	will later be used  to prove Theorem~\ref{th:twist_inf_general}.
	We start by stating an essentially known lemma which can be  used to prove the simple case of Theorem~\ref{th:twist_inf_grass} when $\dim_\C X$ is odd.
	
	\begin{lemma}
		\label{lem:a2_chain_div}
		Let $\cL$ be a very ample line bundle over a K\" ahler manifold $Y$.
		For any $d\ge 3$, every smooth divisor $X\subset Y$ in the linear system $|\cL^{\otimes d}|$
		contains two $|\cL^{\otimes d}|$-vanishing Lagrangian spheres $L_1,L_2$ that intersect transversely, once.\qed
	\end{lemma}
	
	The proposition below should be considered as an equivariant version of Lemma~\ref{lem:a2_chain_div}. It  will be used to
	prove the harder case of Theorem~\ref{th:twist_inf_general} when $\dim_\C X$ is even.
	(When applicable, it in particular provides the conclusion of Lemma~\ref{lem:a2_chain_div} itself. So we will not need to prove Lemma~\ref{lem:a2_chain_div} for our purposes, although the arguments in this section can readily be adopted, in fact simplified, to give such a proof.)
	
	\begin{proposition}
		\label{prop:a2_in_sym_div}
		Let $\cL$ be a very ample line bundle over a K\" ahler manifold $Y$, and
		let $\ii\co Y\to Y$ be a holomorphic involution which lifts to an automorphism of $\cL$.
		Fix $d\ge 3$ and
		let $H^0(Y,\cL^{\otimes d})_\pm$ denote the $\pm 1$-eigenspace of the involution on $H^0(Y,\cL^{\otimes d})$ induced by $\ii$. Let $\Pi_\pm$ be as in Theorem~\ref{th:twist_inf_general}. 
		Pick a  connected component $\tilde\Sigma$ of $Y^\ii\subset Y$, $\dim\tilde\Sigma\ge 2$.
		Suppose one of the following:
		\begin{enumerate}
			\item[(a)] $d$ is even;
			
			\item[(b)] $d$ is odd,
			$\tilde \Sigma\subset \Pi_+$, and
			there is a smooth divisor in the linear system $\P H^0(Y,\cL^{\otimes d})_+$.
		\end{enumerate}
		Then there is a smooth divisor $X$ in the linear system $|\cL^{\otimes d}|$
		and two $|\cL^{\otimes d}|$-vanishing Lagrangian spheres $L_1,L_2\subset X$ such that:
		
		\begin{enumerate}
			\item $\ii(X)=X$,  $\Sigma\coloneqq  X\cap \tilde\Sigma$ is smooth, $\dim\Sigma=\dim\tilde \Sigma-1$
			\item $\ii(L_1)=L_1$, $\ii(L_2)=L_2$;
			\item $L_1,L_2$ intersect transversely, at a single point which belongs to $\Sigma$;
			\item $L_i^\ii=L_i\cap\Sigma$ are Lagrangian spheres in $\Sigma$, $i=1,2$. 
			\item for $i=1,2$ one can choose a symplectomorphism $\tau_{L_i}$ of $X$ representing the Hamiltonian isotopy class of the Dehn twist around $L_i$ such that $\tau_{L_i}$ commutes with $\ii$, and $\tau_{L_i}|_{X^\ii}$ is the Dehn twist around $L_i^\ii$.
		\end{enumerate}
		The same is true if we replace symbols $+$ with $-$ in Case (b).
	\end{proposition}

	\subsection{$A_2$ chains of  Lagrangian spheres from $A_2$ fibrations}

	\begin{definition}[$A_2$ chain of Lagrangian spheres]
		\label{def:A2_chain_Lags}
		Let $X$ be a symplectic manifold. A pair $(L_1,L_2)$ of two Lagrangian spheres  in $X$ is called 
		an $A_2$-chain if $L_1$ and $L_2$ intersect at a single point, and the intersection is transverse.
	\end{definition}
	
	In Section~\ref{sec:van_spheres_intro} we have seen that how to construct Lagrangian spheres as vanishing cycles of Lefschetz fibrations.
	Similarly, one can get $A_2$ chains of Lagrangian spheres from fibrations with slightly more complicated singularities.
	
	\begin{definition}[$A_2$ fibration]
		\label{def:A2_fib}
		Denote by $D\subset \C$ the open unit disk,
		and by $B_\epsilon\subset \C$ the open disk of radius $\epsilon$. Both disks are centered at $0$.
		
		Suppose $E$ is a smooth manifold, $\Omega$ a closed 2-form on $E$ and $\pi\co  E\to D$ is a smooth map. The triple $(E,\Omega,\pi)$ is called
		an $A_2$ fibration if there is a point $p\in E$ (without loss of generality, we assume $\pi(p)=0\in D$), and a neighbourhood $U(p)$ such that:
		\begin{itemize}
			\item all but a finite number of fibres of $\pi$ are regular, and the restriction of $\Omega$ is symplectic on them;
			\item there exists a complex structure on $U(p)$  with a holomorphic chart $x_1,\ldots,x_n$, $x_i\in B_\epsilon$ such that
			$$\pi(x_1,\ldots,x_n)=x_1^2+\ldots+x_{n-1}^2+h(x_n),$$
			where $h(x_n)$ is holomorphic;
			\item $h(x_n)$  has at least 3 roots within $B_{\epsilon/2}$, counted with multiplicities;
			\item for any $x_n\in B_{\epsilon/2}$, $\sqrt{h(x_n)}\in B_{\epsilon/2}$;
			\item $\Omega|_{U(P)}$ is K\"ahler with respect to the above complex structure.
		\end{itemize}
	\end{definition}
	
	\begin{remark}
		The definition allows $\pi$ to have singularities outside of $U(p)$. 
		Also, the definition does not require $p\co E\to D$ to be a proper map, so the smooth fibres $E_t$ need not be symplectomorphic, as we may not be able to integrate the parallel transport vector fields. The generality of this definition is slightly unusual, but it makes no difference to the local construction of $A_2$ chains of Lagrangian spheres, which is the next thing we discuss.
	\end{remark}
	
	In  order to prove Proposition~\ref{prop:a2_in_sym_div}, we need to introduce $A_2$ fibrations with involutions.
	
	\begin{definition}[Involutive $A_2$ fibration]
		\label{def:A2_fib_inv}
		Let $(E,\Omega,\pi)$ be an $A_2$ fibration. It is called an involutive $A_2$ fibration with 
		involution $\ii\co E\to E$ if 
		in the holomorphic chart from Definition~\ref{def:A2_fib} we have in addition:
		$$\ii(x_1,\ldots,x_l,x_{l+1},\ldots,x_n)=(-x_1,\ldots,-x_l,x_{l+1},\ldots, x_n)$$
		for some $l< n$. We denote by $E^\ii$ the fixed locus of $\ii$.
	\end{definition}
	
	\begin{remark}
		It follows from this definition that $\pi|_{E^\ii}\co E^\ii\to D$ is also an $A_2$ fibration.
		Note that $x\in E^\ii$ is regular for $\pi$ if and only if it is regular for $\pi|_{E^\ii}$.
		Indeed, we can decompose $T_xE=T_xE^\ii\oplus N_x$ where $N_x$ is the $(-1)$-eigenspace
		of $d \ii(x)$. Since $\pi \ii=\pi$, $N_x\subset \ker d\pi(x)$. So $\rk d\pi(x)=\rk d\pi(x)|_{T_xE^\ii}$.
		Consequently, for a regular fibre $E_t$, the fixed locus $E_t^{\ii}$ is smooth.
	\end{remark}
	
	The following is a slight refinement of~\cite[Lemma~6.12]{KhoSe02}.
	
	\begin{lemma}
		\label{lem:A2chain_from_fib}
		Let $\pi\co  E\to D$ be an $A_2$ fibration. Then for every  sufficiently small $t\in D$ such that the fibre $E_t\coloneqq \pi^{-1}(t)$
		is smooth, $E_t$ contains an $A_2$ chain of Lagrangian spheres.
	\end{lemma}
	
	We will use the following equivariant analogue of this lemma.
	\begin{lemma}
		\label{lem:A2chain_from_fib_inv}
		Let $\pi\co  E\to D$ be an involutive $A_2$ fibration with an involution $\ii$. 
		Then for every sufficiently small $t\in D$ such that the fibre $E_t\coloneqq \pi^{-1}(t)$
		is smooth, $E_t$ contains an $A_2$ chain of Lagrangian spheres $(L_1,L_2)$ which satisfy
		properties (2)---(5) from Proposition~\ref{prop:a2_in_sym_div}
		with $X\coloneqq E_t$, and $\Sigma$  the connected component of $E_t^\ii$ 
		which is a subset of the connected component of the point $p$ in $E^\ii$. 
	\end{lemma}
	
	\begin{remark}
		Note that $\dim \Sigma=l-1$, where $l$ is the number coming from the co-ordinate chart in Definition~\ref{def:A2_fib_inv}.
	\end{remark}
	
	\begin{proof}[Proof of Lemma~\ref{lem:A2chain_from_fib}]
		Let $U'(p)\subset U(p)$ be the ball around $p$ given by $|x_i|<\epsilon/2$, $i=1,\ldots,n$.
		As in Subsection~\ref{subsec:lef_fib_van_Cyc}, it suffices to assume $\Omega|_{U'(p)}$ is the standard symplectic form
		in the holomorphic chart $(x_1,\ldots,x_n)$ from Definition~\ref{def:A2_fib}. 
		
		The condition that $E_t$ is smooth means the equation $h(x_n)=t$ has no multiple roots
		with $x_n\in B_{\epsilon/2}$. Therefore by Definition~\ref{def:A2_fib}, the equation $h(x_n)=0$ has at least 3 roots with $x_n\in B_{\epsilon/2}$. So for sufficiently small $t$ the equation $h(x_n)=t$ also has at least 3 distinct roots with
		$x_n\in B_{\epsilon/2}$. Pick  three such roots, say $z_1,z_2,z_3\in B_{\epsilon/2}$: $h(z_i)=t$. Let $\gamma_{12}\subset B_{\epsilon/2}$
		be a path from $z_1$ to $z_2$ whose interior avoids the roots of $h-t$.  Define
		$$L_1\coloneqq \bigsqcup_{z\in\gamma_{12}}\{(x_1,\ldots,x_n)\in B_{\epsilon/2}\cap\pi^{-1}(t): \ |x_i|\in \R\cdot \sqrt{-h(z)}\}.$$
		This is a smooth Lagrangian sphere in $\pi^{-1}(t)$  with respect to the restriction of the standard symplectic form on $\C^n$ to $\pi^{-1}(t)$.
		Similarly, let $\gamma_{23}\subset B_{\epsilon/2}\subset \C$
		be a path from $z_2$ to $z_3$ and define $L_2$ by the same formula replacing $\gamma_{12}$ by $\gamma_{23}$.
		If $\gamma_{12}$ and $\gamma_{23}$ are transverse at their common endpoint $z_2$, then $(L_1,L_2)$ is an $A_2$ chain of Lagrangian spheres by \cite[Lemma~6.12]{KhoSe02}.
		Note that $L_1,L_2$ lie in $U'(p)$ by the fourth condition in Definition~\ref{def:A2_fib}.
	\end{proof}

	\begin{proof}[Proof of Lemma~\ref{lem:A2chain_from_fib_inv}]
		We use the notation from the proof  of Lemma~\ref{lem:A2chain_from_fib}.
		Arguing as in that proof $\ii$-invariantly, we can again assume
		$\Omega$ is standard on $U'(p)$. The formulas for $L_1,L_2$ are  invariant under the change $x_i\mapsto -x_i$, $i\le l$, so $L_1,L_2$ are $\ii$-invariant. This proves property~(2)
		from  Proposition~\ref{prop:a2_in_sym_div}.
		Next, we already know $L_1$ intersects $L_2$ transversely at a single point. This point has co-ordinates
		$x_1=0,\ldots,x_{n-1}=0$, $x_n=z_2$. (Recall $z_2$ is a root of $h(x_n)-t$.) This intersection point is $\ii$-invariant, and it obviously belongs to the connected component of the point $p$ in $E^\ii$, so property~(3) from  Proposition~\ref{prop:a2_in_sym_div} holds. Property~(4) is true because
		$E^\ii$ locally around $\pi$ is given by $x_1=\ldots=x_l=0$, and so $L_i\cap\Sigma$ are transverse Lagrangians for the same reason that the $L_i$ are. By their local construction, the $L_i$ do not intersect the connected components of $E_t^\ii$ other than $\Sigma$.
		
		It remains to explain property (5).
		Let $S^{n-1}\subset \R^n$ be the standard unit sphere. Let $\ii_0$ be the involution on $S^n$ which changes the sign of the first $k$ co-ordinates on $\R^n$. It naturally extends to an involution $\ii_0$ on $T^*S^n$.
		It is not hard to check there is an $(\ii,\ii_0)$-equivariant diffeomorphism $V(L_1)\to V(S^n)$
		where $V(L_1)$ is an $\ii$-invariant tubular neighbourhood of $L_1\subset X$ and $V(S^n)$ is an $\ii_0$-invariant
		tubular neighbourhood of the zero-section in $T^*S^n$. Then there is also an $(\ii,\ii_0)$-equivariant symplectomorphism $V(L_1)\to V(S^n)$, by  an equivariant analogue of the Weinstein tubular neighbourhood theorem. The  Dehn twist in $T^*S^n$ is $\ii_0$-equivariant by definition. Its pullback via the equivariant symplectomorphism $V(L_1)\to V(S^n)$ is the desired $\ii$-equivariant Dehn twist inside $E_t$.
	\end{proof}
	
	\subsection{$A_2$ fibrations of divisors from projective embeddings}
	One way of constructing an $A_2$ fibration is to embed all  its fibres $E_t$ as divisors $E_t=X_t\subset Y$ in a single K\"ahler manifold $Y$. This idea can be used to prove Lemma~\ref{lem:a2_chain_div}, and now
	we will run such an argument $\ii$-invariantly to prove Proposition~\ref{prop:a2_in_sym_div}.

	\begin{proof}[Proof of  Proposition~\ref{prop:a2_in_sym_div}]
		Let us recall the setting. We are given a very ample line bundle $\cL\to Y$ over a K\a hler manifold $Y$,
		and  a holomorphic involution $\ii\co Y\to Y$ which lifts to an involution
		on $\cL$. This means $\ii$ induces a linear involution on $H^0(Y,\cL)^*$ splitting it into the direct sum of $\pm 1$ eigenspaces denoted by $H^0(Y,\cL)_\pm^*$. The projectivisations of these eigenspaces are denoted by $\Pi_\pm\subset \P H^0(Y,\cL)^*$. We also denote $\P^N\coloneqq  \P H^0(Y,\cL)^*$, and the $\ii$-induced involution on $\P^N$  by
		$\ii_{\P^N}$.
		The fixed locus of $\ii_{\P^N}$ is
		$\Pi_+\sqcup \Pi_-\subset \P^N$.
		
		Because $\cL$ is very ample, we have an embedding $Y\subset \P^N$, $\cL=\cO_Y(1)\coloneqq \cO_{\P^N}(1)|_Y$, $Y$ is invariant under $\ii_{\P^N}$ and $\ii_{\P^N}|_Y=\ii$, and also
		$$Y^\ii=(Y\cap \Pi_+)\sqcup (Y\cap\Pi_-).$$
		Let $\tilde\Sigma$ be the given connected component of $Y^\ii$ (smooth by assumption), and  $\dim\tilde \Sigma=l$. Then $\tilde \Sigma\subset \Pi_\epsilon$ where $\epsilon$ is one of the two symbols: $+$ or $-$.
		We will also denote by $\epsilon$ the correspondingly signed number $\pm 1$. 
		
		Choose homogeneous co-ordinates $(x_0:\ldots:x_l:x_{l+1}:\ldots:x_{N})$ on $\P^N$
		with the following properties:
		\begin{enumerate}
			\item $\ii_{\P^N}(x_0:\ldots:x_l:x_{l+1}:\ldots:x_{N})=(\epsilon x_0:\ldots:\epsilon x_l:\pm x_{l+1}:\ldots: \pm x_{N+1})$;
			\item $(1:0:\ldots:0)\in \tilde \Sigma$
			\item the plane spanned by $(x_0,\ldots,x_l)$ (other co-ordinates are set to $0$) is the tangent plane to $\tilde\Sigma$ at $(1:0:\ldots:0)$;
			\item for some $n\ge l$, the plane spanned by $(x_0,\ldots,x_n)$ (other co-ordinates are set to $0$) is the tangent plane to $Y$ at $(1:0:\ldots:0)$.
		\end{enumerate}
		The third property implies that
		$x_0,\ldots,x_l$, seen as sections in $H^0(\cO_{\P^N}(1))$,
		belong to the $\epsilon$-eigenspace of  $\ii$. This is in agreement with the first property.
		So  co-ordinates with the above properties exist.
		
		\begin{figure}[h]
			\centerline{\includegraphics{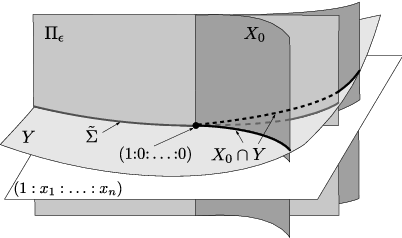}}
			\caption{A divisor $X_0$ from the family $X_t$ constructed in the proof of Proposition~\ref{prop:a2_in_sym_div}.}
			\label{fig:divisors_fig}
		\end{figure}

		In the affine chart $x_0=1$,
		the co-ordinates $(x_1,\ldots,x_n)$ serve as  local co-ordinates for $Y$ near the origin.
		In the chart $x_0=1$, write (see Figure~\ref{fig:divisors_fig}):
		$$
		X_t\coloneqq x_1^3+x_2^2+\ldots+x_n^2-t.
		$$
		We want $X_t$ to be a section of $\cO_{\P^N}(d)$, so in projective co-ordinates we set
		$$
		X_t\coloneqq x_0^{d-3}x_1^3+x_0^{d-2}(x_2^2+\ldots+x_n^2) -tx_0^{d}.
		$$
		From property~(1) of the co-ordinates $x_i$,
		we see that $X_t\circ \ii=\epsilon^{d} X_t$
		as polynomials. In other words:
		\begin{enumerate}
			\item[(a)] if $d$ is even, $X_t\in H^0(\cO_{\P^N}(d))_+$;
			
			\item[(b)] if $d$ is odd, $X_t\in H^0(\cO_{\P^N}(d))_\epsilon$.
		\end{enumerate}
		
		For all $t$, the divisors $\{X_t=0\}$ and $\{X_t=0\}\cap Y$ are reducible and hence singular. We want to smooth 
		the family $\{X_t=0\}\cap Y$ so that a generic divisor in this $t$-family becomes non-singular.
		
		Suppose $d$ is even.
		Then the linear   system $H^0(\cO_{\P^N}(d))_+$ has no base locus as it contains all monomials
		$x_i^d$. Then $H^0(\cO_Y(d))_+=H^0(Y,\cL^{\otimes d})_+$ has no base locus too.
		By Bertini's theorem in characteristic $0$, 
		there exists $F\in H^0(\cO_{\P^n}(d))_+$  such that the divisor $\{F=0\}\cap Y$ is smooth.
		
		Suppose $d$ is odd. Then the linear systems $H^0(\cO_{\P^N}(d))_\pm$ have non-empty base loci, namely $\Pi_\mp$ (see the proof of Lemma~\ref{lem:exist_smooth_inv_div} below).
		Therefore it is not a priori clear that these linear systems contain a smooth divisor.
		This condition is included in the assumptions of Proposition~\ref{prop:a2_in_sym_div}, Case~(b).
		Let $ F\in H^0(\cO_{\P^n}(d))_\epsilon$ be a polynomial such that $\{F=0\}\cap Y$ is smooth. 
		
		The rest of the proof is the same for even and odd $d$. For all generic $\delta\in \C$, the divisors $\{X_t+\delta F=0\}\cap Y$ are smooth except for a finite number of $t$'s.
		Recall that  $(x_1,\ldots,x_n)$ is a holomorphic chart for $Y$ around $(1:0:\ldots:0)$.
		There is another chart $\tilde x_1,\ldots,\tilde x_n$
		in which the divisors $\{X_t+\delta F=0\}\cap Y$ are given by:
		$$
		h(\tilde x_1)+\tilde x_2^2+\ldots+\tilde x_n^2- t+c=0
		$$
		where $h(\tilde x_1)$ is close to $\tilde x_1^3$ (when $\delta$ is small) and $c$ is a small constant.
		Moreover, the change of co-ordinates from $x_i$ to $\tilde x_i$ is $\ii$-equivariant.
		This follows from an equivariant version of the holomorphic Morse splitting lemma \cite{Atiyah58}.
		
		Consider the family
		$\{X_t+\delta F=0\}\cap Y$ of divisors in $Y$, $t\in D$. They are  $\ii$-invariant
		and belong to the linear system $|\cL^{\otimes d}|$.
		Let $E\to D$ be the total space of this family, see  Definition~\ref{def:tot_fam_div}.
		It may be singular; if it is, remove its singular locus to get $E_0$.
		The involution $\ii$ turns $E_0\to D$ into an involutive fibration in the sence of Definition~\ref{def:A2_fib_inv}. So by Lemma~\ref{lem:A2chain_from_fib_inv}, a smooth divisor in the family
		$\{X_t+\delta F=0\}\cap Y$ has a pair of Lagrangian spheres $(L_1,L_2)$ that satisfy properties (2)---(5)
		of Proposition~\ref{prop:a2_in_sym_div}. 
		It is easy to see that  Lemma~\ref{lem:A2chain_from_fib_inv} constructs $L_1,L_2$ which are $|\cL^{\otimes d}|$-vanishing.
		
		It remains to check $\iota$ satisfies property (1). We have to show that the smooth divisors $\{X_t+\delta F=0\}\cap Y$ intersect $\Sigma=\tilde\Sigma \cap Y$ transversely.  Suppose $X\coloneqq \{X_t+\delta F=0\}\cap Y$
		intersects $\Sigma$ non-transversely at one point $p$, so $T_p\Sigma\subset T_p X$ (the tangent spaces are taken inside $Y$). This means $T_pX$ contains $\dim \Sigma$ positive ($+1$) eigenvalues of $d\ii$. Then the same must hold for all intersection points $X\cap \Sigma$, and hence $T_p\Sigma\subset T_p X$ for any $p\in X\cap \Sigma$.
		But in a neighbourhood of $(1:0:\ldots:0)$ the intersection $X\cap\Sigma$ is transverse, which is easily verified in
		the local chart $(x_1,\ldots, x_n)$ from above. So $X$ intersects $\Sigma$ transversely everywhere.
		Similarly, every other connected component of $Y^\ii$ either intersects $X$ transversely or is contained in $X$.
	\end{proof}

\end{section}

\section[Proofs of the theorems about Lagrangian spheres]{Proofs of the theorems about Lagrangian spheres in divisors}
\label{sec:proofs_spheres_divs}

\begin{proof}[Proof of Theorem~\ref{th:twist_inf_general}]
	Apply Proposition~\ref{prop:a2_in_sym_div} to $Y,\cL,\tilde \Sigma$ given by the hypothesis of Theorem~\ref{th:twist_inf_general}.
	Proposition~\ref{prop:a2_in_sym_div} returns an $|\cL^{\otimes d}|$-divisor $X\subset Y$ and $|\cL^{\otimes d}|$-vanishing Lagrangian spheres $L_1,L_2\subset X$ satisfying the conditions enumerated there.
	Because  $|\cL^{\otimes d}|$-vanishing spheres are unique up to symplectomorphism (Lemma~\ref{lem:van_unique}), it suffices to show that $\tau_{L_1}$ has infinite order in $\Symp(X)/\Ham(X)$.
	To show this, we compute the Lefschetz number of $\tau_{L_1}^{2k}\tau_{L_2}^{2k}|_{X^\ii}=\tau_{L_1^\ii}^{2k}\tau_{L_2^\ii}^{2k}$ on $H^*(X^\ii)$, where $X^\ii$ is the fixed locus of the involution $\ii$ on $X$. Recall that $\Sigma=\tilde \Sigma\cap X$ is a connected component of $X^\ii$. We are given that $\dim\tilde \Sigma$ is even, so $\dim\Sigma=\dim \tilde \Sigma-1$ is odd. 
	Let $X^\ii=\Sigma\sqcup \Sigma_0$ where $\Sigma_0$ is all other connected components.
	We  identify $H^*(X^\ii)$ with $H_*(X^\ii)$ via Poincar\'e duality.
	
	Consider the homology classes $[L_1^\ii],[L_2^\ii]\in H_*(\Sigma)$, and recall that $[L_1^\ii]\cdot [L_2^\ii]=1$.
	Denote  $s=\dim_\C\Sigma$ and $\epsilon=(-1)^{\frac 1 2 s(s-1)}$.
	Using the Picard-Lefschetz formula (see Subsection~\ref{subsec:dehn_twists}) and property~(5) from Proposition~\ref{prop:a2_in_sym_div}, let us write down the actions of the Dehn twists on the 2-dimensional vector space spanned by $\{ [L_1^\ii],[L_2^\ii]\}\subset H_*(X^\ii)$: 
	$$
	(\tau_{L_1^\ii})_*^{2k}:
	\left(
	\begin{matrix}
	1 & k(1{+}({-}1)^{s{-}1})\epsilon\\
	0 & 1            
	\end{matrix}
	\right),
	\qquad
	(\tau_{L_2^\ii})_*^{2k}:
	\left(
	\begin{matrix}
	1 & 0\\
	k\left(1{+}({-}1)^{s{-}1}\right)\epsilon& 1
	\end{matrix}
	\right).
	$$
	Now since $s=\dim_\C \Sigma$ is odd, we see that
	$$
	STr \left((\tau_{L_1^\ii})_*^{2k}(\tau_{L_2^\ii})_*^{2k}|_{\spn\{ [L_1^\ii],[L_2^\ii]\}}\right)=-4k^2 -2.
	$$
	(The negative signs appear because we are computing the supertrace). If $s$ were even, we would get the constant $2$ instead.
	
	We can extend $[L_1^\ii],[L_2^\ii]$ to a basis of $H_*(X^\ii)$ in which all other elements have zero intersection with $[L_1^\ii],[L_2^\ii]$. By the Picard-Lefschetz formula, $(\tau_{L_i^\ii})_*$ acts by $\id$ on the rest of such basis.
	Consequently, the Lefschetz number is
	$$
	L \left((\tau_{L_1^\ii})^{2k}(\tau_{L_2^\ii})^{2k}\right)=- 4k^2+c,
	$$
	where $c$ is a constant independent of $k$.
	By Proposition~\ref{prop:hf_bound_fixp_notrans},
	\begin{equation}
	\label{eq:growth_hf}
	\dim_\Lambda HF^*(\tau_{L_1}^{2k}\tau_{L_2}^{2k})\ge |-4k^2+c|.
	\end{equation}

	Suppose $\tau_{L_1}^{2k}$ is Hamiltonian isotopic to $\id$ for some $k>0$. Then  $\tau_{L_2}^{2k}$ is also Hamiltonian isotopic to $\id$, because by Lemma~\ref{lem:van_unique} there is a symplectomorphism of $X$ taking $L_1$ to $L_2$.
	Then the product $\tau_{L_1}^{2k}\tau_{L_2}^{2k}$ is also Hamiltonian isotopic to $\id$. Since $k$ can be taken arbitrarily large, this contradicts to the growth of dimensions in Equation~(\ref{eq:growth_hf}).
	Consequently 
	$\tau_{L_1}$ has infinite order in the group $\Symp(X)/\Ham(X)$.
\end{proof}

Next we prove Lemma~\ref{lem:exist_smooth_inv_div}.
It follows from a strong Bertini theorem which we now quote.

\begin{theorem}[{\cite[Corollary~2.4]{DiHa91}}]
	\label{th:strong_bertini}
	Let $Y$ be a compact smooth complex manifold and $S$ an effective linear system of divisors on $Y$.
	Let  $B$ be the base locus of $S$. If $B$ is reduced and non-singular, and $\dim B<\half \dim Y$, then a generic divisor in $S$
	is smooth.\qed
\end{theorem}
If $B$ is disconnected, the dimensional inequality must hold for every connected component of $B$.

\begin{proof}[Proof of Lemma~\ref{lem:exist_smooth_inv_div}]
	We repeat the beginning of the proof of Proposition~\ref{prop:a2_in_sym_div}.
	We have
	$Y\subset \P^N$
	and $\cL^{\otimes d}=\cO_{\P^N}(d)|_Y$.
	The involution $\ii$ acts on sections of $\cL$ and so acts on $\P^N$ by a linear involution $\ii_{\P^N}$,
	and $Y\subset \P^N$ is invariant under it.
	Pick homogeneous co-ordinates $(x_0:\ldots:x_N)$
	such that
	$$
	\ii_{\P^N}(x_0:\ldots x_l:x_{l+1}:\ldots:x_N)=(x_0:\ldots:x_l:-x_{l+1}:-x_N).
	$$
	Recall that $d$ is odd by assumption. Then $H^0(\cO_{\P^N}(d))_+$ consists of degree-$d$ polynomials which are sums of monomials of the following form:
	$$
	x_0^{odd}\ldots x_l^{odd}x_{l+1}^{even}\ldots x_N^{even}.
	$$
	Here {\it even} or {\it odd} denote the parity of a power.
	The base locus of the linear system  \linebreak[4] $\P H^0(\cO_{\P^N}(d))_+$
	is given by
	$$
	x_0=0,\ \ldots,\ x_l=0
	$$
	and so coincides with $\Pi_-$. The base locus $B$ of $\P H^0(Y,\cL^{\otimes d})_+$ is therefore $\Pi_-\cap Y$. It is smooth because $Y^\ii$ is smooth.
	We are also given that $\dim B< \half \dim Y$ by hypothesis.
	Finally, we know that $\ii_{\P^N}|_{\Pi_-}=\id$, so $Y$ intersects $\Pi_-$ cleanly (i.e.~transversely in the normal direction to $\Pi_-\cap Y$), and hence $B=\Pi_-\cap Y$ is reduced.
	Consequently, Lemma~\ref{lem:exist_smooth_inv_div} follows from Theorem~\ref{th:strong_bertini}.
	(The case when the signs symbols $+$ and $-$ are interchanged is analogous.)
\end{proof}

We now return to divisors  in Grassmannians and prove Theorem~\ref{th:twist_inf_grass}.
Let $Gr(k,n)\subset \P^N$ be the Pl\" ucker embedding; the anti-canonical class of $Gr(k,n)$
equals $\cO_{\P^N}(n)|_{Gr(k,n)}$
\cite[Proposition 1.9]{Muk93}.
Consequently, a smooth divisor  $X\subset Gr(k,n)$ in the linear system
$\cO_{\P^N}(d)|_{Gr(k,n)}$ satisfies the $W^+$ condition, see Definition~\ref{def:weak_monot},
if and only if 
$
d\le n$ or $d\ge k(n-k)+n-2$,
and $X$ is monotone (Fano) if and only if $d< n$.

\begin{proof}[Proof of Theorem~\ref{th:twist_inf_grass}]
	We have already mentioned this theorem is easy and essentially known when 
	$k(n-k)$ is even.
	(The sphere $L\subset X$ is non-trivial in $H_n(X)$
	by Lemmas~\ref{lem:a2_chain_div} and~\ref{lem:van_unique}. Then apply Corollary~\ref{cor:twist_homol_order}(2).)
	We will now prove the hard case when $k(n-k)$ is odd using the general Theorem~\ref{th:twist_inf_general}.
	Denote $k=2p+1$, $n=2q$.
	
	Consider a linear involution on $\C^{2q}$ with $q+l$ positive eigenvalues and $q-l$ negative eigenvalues for some $l$.
	It induces a non-degenerate involution $\ii$ on $Gr(2p+1,2q)$ whose fixed locus is
	$$
	Gr(2p+1,2q)^\ii=\bigsqcup_{t=0}^{2p+1} Gr(t,q+l)\times Gr(2p+1-t,q-l).
	$$
	This fixed locus consists of $(2p+1)$-planes that admit a frame in which $t$ vectors lie in the positive eigenspace of the involution on $\C^{2q}$, and the remaining $2p+1-t$ vectors lie in the negative eigenspace.
	We compute:
	\begin{multline}
	\label{eq:dims_grass}
	\dim Gr(t,q+l)+\dim  Gr(2p+1-t,q-l)
	\\ -\half \dim Gr(2p+1,2q)
	=-\half(1+2p-2t)(1+2p+2l-2t).
	\end{multline}
	
	For this paragraph, set $l=0$. Then the expression (\ref{eq:dims_grass}) is less than $0$ for any $t\in \Z$.
	This means  $\dim Gr(2p+1,2q)^\ii < \half \dim Gr(2p+1,2q)$. (The left-hand side is disconnected, and we mean that the inequality holds for each of its connected components.)
	Therefore we can apply Lemma~\ref{lem:exist_smooth_inv_div} to either of the two linear systems $\P H^0(Y,\cL^{\otimes d})_\pm$. In order to apply Theorem~\ref{th:twist_inf_general}, it remains to check that $Gr(2p+1,2q)^\ii$ contains a connected component of even dimension.
	A computation shows that a connected component  of $Gr(2p+1,2q)^\ii$ has dimension of parity
	$$\dim Gr(t,q)+\dim  Gr(2p+1-t,q)\equiv q-1 \mod 2$$
	independently of $t$. We will now consider the case when $q$ is odd, and will discuss the case when $q$ is even in the next paragraph. If $d$ is odd, apply Theorem~\ref{th:twist_inf_general}(b) taking either of the two sign symbols $+$ or $-$.
	If $d$ is even, apply Theorem~\ref{th:twist_inf_general}(a) (this case is easier and does not require the computation of dimensions we have made).
	This proves Theorem~\ref{th:twist_inf_grass} for $Gr(2p+1,2q)$ in the case when $q$ is odd.
	
	Now suppose  $q$ is even.
	Set $l=1$ until the end of the proof.
	Recall that $Gr(2p+1,2q)^\ii=(\Pi_+\sqcup \Pi_-)\cap Gr(2p+1,2q)$.
	The only case when (\ref{eq:dims_grass}) fails to be less than zero is when
	$$
	1+2p-2t=-1.
	$$
	This happens for a unique $t\in \Z$.
	So either $\dim Gr(2p+1,2q)\cap \Pi_+<\half \dim Gr(2p+1,2q)$, or the same holds with $\Pi_-$ taken instead. (As above, we mean that the inequality holds for each  connected component of the left hand side.)
	A computation shows that a connected component of $Gr(2p+1,2q)^\ii$ has dimension of parity
	$$\dim Gr(t,q+1)+\dim  Gr(2p+1-t,q-1)\equiv q \mod 2\equiv 0 \mod 2$$
	Therefore we can apply
	Lemma~\ref{lem:exist_smooth_inv_div} and Theorem~\ref{th:twist_inf_general} taking that symbol $+$ or $-$ for which the inequality $\dim Gr(2p+1,2q)\cap \Pi_\mp <\half \dim Gr(2p+1,2q)$
	holds. Theorem~\ref{th:twist_inf_grass} is proved in all cases.
\end{proof}

\begin{proof}[Proof of Corollaries~\ref{cor:fund_gp_grass},~\ref{cor:fund_gp_general}.]
	These corollaries follow from Theorems~\ref{th:twist_inf_grass},~\ref{th:twist_inf_general} and Lemma~\ref{lem:dehn_tw_and_monodromy}.
\end{proof}




\section{Growth of Lagrangian Floer cohomology and ring structures}
\label{app:growth_ring}

\subsection{Dehn twists around spheres with deformed cohomology}
The main theorems of this chapter have been proved; this last section is devoted to an additional observation on the relation between the Floer cohomology of a Lagrangian sphere and its associated Dehn twist.
Keating \cite{Ke14}  has recently obtained  an exact sequence  involving iterated Dehn twists in the Fukaya category of a symplectic manifold, extending  Seidel's original exact sequence \cite{Sei03_LES}.
In this subsection we use it to prove Proposition~\ref{prop:growth_and_ring}, which is stated below.
Then we apply it to compute Floer cohomology rings of vanishing spheres
in some divisors.

Let $X$ be a compact monotone symplectic manifold.
Denote by $\F(X)$ its monotone Fukaya category over $\C$, which is a collection of \ai categories $\F(X)_\lambda$, $\lambda\in \C$, corresponding to the eigenvalues of  multiplication with $c_1(X)$ in $QH^*(X)$. 
Our aim is to prove the following.

\begin{proposition}
	\label{prop:growth_and_ring}
	Let $X$ be a monotone symplectic manifold, $\dim_\R X=4k$ for some $k\ge 1$,  $L_1\subset X$ be a Lagrangian sphere
	and $L_2\subset X$ another monotone Lagrangian which intersects $L_1$ transversely, once. Assume $L_1,L_2$ are included into the same summand $\F(X)_\lambda$. Suppose that $\dim HF^*(\tau_{L_1}^k L_2,L_2)>2$ for some $k\in \N$. Then there is an isomorphism of rings $HF^*(L_1,L_1)\cong \C[x]/x^2$.
\end{proposition}

We will use the language of \ai categories and refer to~\cite{SeiBook08} for the relevant definitions.
All \ai algebras and modules in this section are assumed to be minimal.


\begin{definition}
	\label{def:TruncBar}
	Let $A$ be a strictly unital $\Z/2$-graded \ai algebra  with unit $1\in A$,
	$M$  a right \ai module over $A$ and $N$ a left \ai module over $A$.
	Fix an augmentation, i.e.~a vector space splitting $A=(1)\oplus \bar A$. 
	The $k$-truncated bar complex is the vector space
	$$(M\otimes_AN)_k\coloneqq 
	\bigoplus_{j=0}^{k-1}M\otimes \bar A^{\otimes j}\otimes N
	$$
	with the differential that on the $j$th summand equals
	\begin{equation}
	\label{eq:bar_complex}
	\sum_
	{\begin{smallmatrix}
		j+2=p+q+r, \\
		p,\, r\ge 0,\ q\ge 2
		\end{smallmatrix}}
	(-1)^{\maltese}(-1)^{r}(\id^{\otimes p}\otimes \mu^q\otimes \id^{\otimes r}).
	\end{equation}
	Here $\maltese\in\{0,1\}$ depends on the gradings of the arguments: if the input is $m\otimes x_1\otimes\ldots\otimes x_{k-1}\otimes n$, where $m\in M$, $x_i\in A$, $n\in N$, then $\maltese$ is the sum of gradings of the last $r$ elements of the input.
	If we put $p=0$ in (\ref{eq:bar_complex}), we get the summand
	$\mu^q\otimes \id^{\otimes r}$
	which involves the module structure map $\mu^q\co M\otimes A^{\otimes (q-1)}\to M$.
	Similarly, when we put $r=0$ in (\ref{eq:bar_complex}), 
	$\mu^q$ is understood to be the module structure map
	$\mu^q\co A^{\otimes (q-1)}\otimes N\to N$.
	When $p,r>0$, $\mu^q$ denotes the algebra structure map
	$A^{\otimes q}\to A$ composed with the augmentation $A\to\bar A$.
\end{definition}

\begin{theorem}[Keating, {\cite[Lemma 7.2 and Remark 6.6]{Ke14}}]
	\label{th:KeatingCone}
	Suppose  $L_1,L,L_2\subset X$ are three Lagrangian submanifolds which are objects of $\F(X)_\lambda$, and $L$ is a sphere. 
	Then there is an exact sequence of vector spaces below. \qed
	
	$$
	\xymatrix @C=-5pc{
		HF^*(L_1,L_2)\ar[rr]&& HF^*(\tau^k_LL_1,L_2)\ar[dl]\\
		&H\left( Hom(L,L_1)\otimes_{Hom(L,L)}Hom(L_2,L)\right)_k
		\ar[ul]}
	$$
\end{theorem}
Here the $Hom$-spaces denote Floer complexes seen as the morphism spaces of the Fukaya category; for example,
$Hom(L,L)=CF^*(L,L)$ has an \ai algebra structure whose definition was sketched in Chapter~\ref{ch:0}.

Note that \cite{Ke14} states  this theorem  for exact $X$ and over $\Z/2$; in particular, it does not mention the signs in (\ref{eq:bar_complex}). The proof uses a theorem of Seidel~\cite[Corollary~17.17]{SeiBook08} which says that $\tau_LL_1$ is quasi-isomorphic to the cone of a certain evaluation map, as an object of the (category of twisted complexes over the) Fukaya category. This allows to write $\tau^k_L L_1$ as an iterated cone, which automatically provides some exact sequence of the type above. Keating proves Theorem~\ref{th:KeatingCone}  by simplifying the iterated cone in a purely algebraic way: by identifying and killing some acyclic sub-complexes in it. We know that the initial Seidel's theorem holds for the monotone Fukaya category and over $\C$ (see e.g.~Oh \cite{Oh11} for the homological version), and the proof of Theorem~\ref{th:KeatingCone}  works in the monotone case and over $\C$ 
by virtue of being purely algebraic. The signs in (\ref{eq:bar_complex}) will be enforced for algebraic reasons, and it is a matter of book-keeping to check that they are the ones that we expect to see in a bar complex.
In addition to Theorem~\ref{th:KeatingCone}, we will
need some auxiliary lemmas.

\begin{lemma}[Formality]
	\label{l:deformed_formal}
	Every \ai algebra whose cohomology ring is $\C[x]/(x^2-1)$
	is quasi-isomorphic
	to the \ai algebra $\C[x]/(x^2-1)$ with vanishing higher multiplications: $\mu^j=0$, $j>2$.
\end{lemma}

\begin{proof}
	The Hochschild cohomology of the associative algebra $\C[x]/(x^2-1)$ is concentrated in degree 0; this is proved in \cite[Proposition~2.2]{Holm01} when $x$ has even degree and in \cite{Ka86} when $x$ has odd degree. The lemma then follows from  \cite[Corollary~4]{Kade88}; see also \cite[Section~3]{Sei15}.
\end{proof}

\begin{lemma}
	\label{lem:formal_modules}
	Take the \ai algebra $\C[x]/(x^2-1)$ with vanishing $\mu^j, j > 2$. Every
	strictly unital \ai module $M$ over this algebra with vanishing $\mu^1$ necessarily has
	vanishing $\mu^j, j>2$.
\end{lemma}

\begin{proof}
	Take the minimal $j$ such that $\mu^j(m, x^{\otimes (j-1)})\neq
	0$ for some $m\in M$. If $j > 1$, the \ai relation for the tuple $(m, x^{\otimes (j-1)}, 1)$ gives
	$\mu^j(m, x^{\otimes (j-1)}) = 0$, a contradiction.
\end{proof}

\begin{lemma}[{\cite[Lemma 3.1]{Ke14}}]
	\label{lem:make_strict_unit} Let $(M,A,N)$ be a c-unital  \ai category consisting of an \ai algebra $A$, a left \ai module $M$ and a right \ai module $N$. Let $A'$ be a strictly unital \ai algebra quasi-isomorphic to $A$. Then there are strictly unital \ai modules $M',N'$ over $A'$ such that the category $(M,A,N)$ is quasi-isomorphic to $(M',A',N')$. The underlying Hom spaces of $(M,A,N)$ and $(M',A',N')$ are the same.
	\qed
\end{lemma}

\begin{lemma}[{Cf.~\cite[Lemma 7.3]{Ke14}}] 
	\label{lem:replace_bar}
	Let $(M,A,N)$ and $(M',A',N')$ be two strictly unital \ai categories consisting of an algebra, a left module and a right module. If they are quasi-isomorphic, the associated bar complexes
	$(M\otimes_A N)_k$ and $(M'\otimes_{A'}N')_k$ are quasi-isomorphic.
	\qed
\end{lemma}


\begin{remark}
	\label{rem:hf_ring}
	Let $\dim_\R X=2n$. Suppose $L\subset X$ is a Lagrangian sphere.
	The $\Z/2$-graded Floer chain complex
	$CF^*(L,L)$ can be realised as a 2-dimensional vector space $\C\oplus \C$
	with two generators: the unit $1$, $\deg 1=0$ and the second generator $x$, $\deg x\equiv n \mod 2$. The differential has degree $1$.
	If $n$ is even, Floer's differential must vanish and  $HF^*(L,L)$ is a unital 2-dimensional commutative algebra.
	Up to isomorphism, this leaves only two possibilities: $\C[x]/x^2$ or $\C[x]/(x^2-1)$.
	If $n$ is odd, $HF^*(L,L)$ can also vanish.
	
	The minimal Chern number of $X$ is the maximal integer $N$ such that $c_1(X)$ is divisible by $N$
	in integral cohomology $H^2(X;\Z)$. The Floer cohomology of a Lagrangian sphere can be made $\Z/2N$ graded,
	and our generators have gradings $\deg 1=0$, $\deg x\equiv n\mod 2N$.
	If $n\neq 0\mod N$, for grading reasons we obtain $x^2=0$ and $HF^*(L,L)\cong \C[x]/x^2$.
\end{remark}

\begin{proof}[Proof of Proposition~\ref{prop:growth_and_ring}]
	We want to prove that $HF^*(L_1,L_1)\cong \C[x]/x^2$.
	Suppose this is not the case, then by Remark~\ref{rem:hf_ring},
	$HF^*(L_1,L_1;\C)\cong \C[x]/(x^2-1)$. Recall that $n$ is even.
	
	Inside $\F(X)_\lambda$, take the subcategory consisting of the \ai algebra
	$Hom(L_1,L_1)$, its left module $Hom(L_1,L_2)$ and its right module $Hom(L_2,L_1)$.
	Because $|L_1\cap L_2|=1$, $Hom(L_1,L_2)$ and $Hom(L_2,L_1)$ are 1-dimensional as vector spaces.
	Denote their generators by
	$$
	Hom(L_1,L_2)=\langle m\rangle,\quad Hom(L_2,L_1)=\langle n\rangle.
	$$
	
	By Lemma~\ref{l:deformed_formal}, the \ai algebra $Hom(L_1,L_1)$ is quasi-isomorphic to the associative algebra $\C[x]/(x^2-1)$ with trivial higher multiplications.
	By Lemma~\ref{lem:make_strict_unit} and Lemma~\ref{lem:formal_modules},
	modules $Hom(L_1,L_2)$ and $Hom(L_2,L_1)$ are quasi-isomorphic to those with trivial higher multiplications over $\C[x]/(x^2-1)$.
	The module $\mu^2$-operations, however, must be non-trivial because $x^2=1$: 
	$$\mu^2(m,x)=\epsilon_m m,\quad \mu^2(x,n)=\epsilon_n n\qquad \text{where}\quad \epsilon_m,\epsilon_n=\pm 1.$$  
	Lemma~\ref{lem:replace_bar} allows to compute the homology of the bar complex
	$$B_k\coloneqq \left( Hom(L_1,L_2)\otimes_{Hom(L_1,L_1)}Hom(L_2,L_1)\right)_k$$
	using the simple associative model we obtained.
	In this model, the bar complex $B_k$ is based on the $k$-dimensional vector space
	$$
	\bigoplus_{j=1}^{k-1} m\otimes x^{\otimes j}\otimes n.
	$$
	The differential comes only from $\mu^2(m,x)$ and $\mu^2(x,n)$:
	$$
	\bd (m\otimes x^{\otimes j}\otimes n)=
	\left((-1)^{j}\epsilon_n+\epsilon_m\right)
	m\otimes x^{\otimes (j-1)}\otimes n.
	$$
	Note that  $(-1)^\maltese=1$ because we are given $\deg x=0$ and may assume $\deg n=0$.
	We see that $\dim H(B_k)=0$ or $1$, depending on the parity of $k$.
	By the exact sequence of Theorem~\ref{th:KeatingCone}, we get $\dim HF^*(\tau_{L_1}^k L_2,L_2;\C)\le 2$,
	which contradicts to the hypothesis.
\end{proof}

\begin{remark}
	If $HF^*(L_1,L_1;\C)\cong \C[x]/x^2$, it might still happen that $Hom(L_1,L_1)$ is formal, for example  when $X$ is exact.
	Running the above proof, from $x^2=0$ we conclude that
	$\mu^2(m,x)=\mu^2(x,n)=0$. So the differential on the
	$k$-dimensional model for $B_k$ written above vanishes, and $\dim H(B_k)=k$.
	This agrees with the growth of $\dim HF^*(\tau_{L_1}^k L_2,L_2)$.
\end{remark}

\subsection{Floer cohomology rings of Lagrangian spheres in divisors}
We now combine Proposition~\ref{prop:growth_and_ring} with previous results (Propositions~\ref{prop:hf_bound_lag} and~\ref{prop:a2_in_sym_div}) to compute the ring
$HF^*(L,L;\C)$ for vanishing Lagrangian spheres $L$ in certain divisors; we use the notation
from Subsection~\ref{subsec:gen_theorems}. We also provide a corollary which specialises  to divisors in Grassmannians.

\begin{proposition}
	\label{prop:ring_div}
	In addition to the conditions of Theorem~\ref{th:twist_inf_general} (a) or (b), suppose
	$X$ is Fano and $\dim_\C X$ is even. Then there is a ring isomorphism $HF^*(L,L;\C)\cong \C[x]/x^2$.
\end{proposition}

\begin{corollary}
	\label{cor:ring_grass}
	Let $X\subset Gr(k,n)$ be a smooth divisor of degree $3\le d < n$, $\dim_\C X$ even.
	Let $L\subset X$ be an $|\cO(d)|$-vanishing Lagrangian sphere. Then there is a ring isomorphism  $HF^*(L,L;\C)\cong \C[x]/x^2$.
\end{corollary}

The possibility ruled out by these two statements is the deformed ring $HF^*(L,L)\cong \C[x]/(x^2-1)$.
An example of a sphere with $HF^*(L,L)\cong \C[x]/(x^2-1)$
is the antidiagonal $L\subset \P^1\times \P^1$.
Note that for this sphere, $\tau_L$ has order 2 in $\pi_0\Symp(\P^1\times \P^1)$ \cite{SeiL4D}.
It seems natural to ask whether there is a general relation between the isomorphism $HF^*(L,L)\cong\C[x]/(x^2-1)$ and
$\tau_L$ being of finite order (both cases are rare). 
Observe that
for many, but not all, pairs $(k,n)$ Corollary~\ref{cor:ring_grass} follows the grading consideration in Remark~\ref{rem:hf_ring}.

\begin{proof}[Proof of Proposition~\ref{prop:ring_div}]
	
	As in the beginning of the proof of Theorem~\ref{th:twist_inf_general}, take $X,L_1,L_2$
	constructed in Proposition~\ref{prop:a2_in_sym_div}. By Lemma~\ref{lem:van_unique},
	it suffices to prove that $HF^*(L_1,L_1)\cong \C[x]/x^2$.
	
	From the Picard-Lefschetz formula (Lemma~\ref{lem:picard_lef}),
	given that $|L_1\cap L_2|=1$ and $\dim L_i^\ii$ is odd, we get the equality 
	$[\tau_{L_1^\ii}^k L_2^\ii]=[L_2^\ii ]-\epsilon k [L_1^\ii]$ in the homology of the fixed locus $H_*(X^\ii)$.
	Consequently, 
	$[\tau_{L_1^\ii}^k L_2^\ii]\cdot [L_2^\ii]=-\epsilon k$.
	By Proposition~\ref{prop:hf_bound_lag},
	$\dim HF^*(\tau_{L_1}^k L_2,L_2)\ge k$.
	By Proposition~\ref{prop:growth_and_ring},
	$HF^*(L_1,L_1)\cong \C[x]/x^2$.
\end{proof}

\begin{proof}[Proof of Corollary~\ref{cor:ring_grass}]
	Repeat the proof of Theorem~\ref{th:twist_inf_grass} but refer to Proposition~\ref{prop:ring_div} instead of Theorem~\ref{th:twist_inf_general}. 
	Recall the condition $d<n$  means that $X$ is Fano.
\end{proof}


\bibliography{Symp_bib}{}
\bibliographystyle{plain}

\end{document}